\documentclass[11pt]{article}

\usepackage[margin=1in]{geometry}
\usepackage{amsmath,amssymb,amsfonts,amsthm,mathtools}
\usepackage{bm}
\usepackage{enumitem}
\usepackage{booktabs}
\usepackage{microtype}
\usepackage{hyperref}
\usepackage{placeins}
\usepackage[nameinlink,capitalize]{cleveref}
\usepackage[authoryear,round]{natbib}

\hypersetup{colorlinks=true,linkcolor=blue,citecolor=blue,urlcolor=blue}

\newtheorem{theorem}{Theorem}[section]
\newtheorem{proposition}[theorem]{Proposition}
\newtheorem{lemma}[theorem]{Lemma}
\newtheorem{corollary}[theorem]{Corollary}
\newtheorem{assumption}[theorem]{Assumption}
\newtheorem{remark}[theorem]{Remark}

\crefname{assumption}{Assumption}{Assumptions}
\Crefname{assumption}{Assumption}{Assumptions}
\crefname{proposition}{Proposition}{Propositions}
\Crefname{proposition}{Proposition}{Propositions}
\crefname{theorem}{Theorem}{Theorems}
\Crefname{theorem}{Theorem}{Theorems}
\crefname{corollary}{Corollary}{Corollaries}
\Crefname{corollary}{Corollary}{Corollaries}
\crefname{lemma}{Lemma}{Lemmas}
\Crefname{lemma}{Lemma}{Lemmas}
\crefname{remark}{Remark}{Remarks}
\Crefname{remark}{Remark}{Remarks}

\newcommand{\E}{\mathbb{E}}

\title{Weighted Nuclear Elastic Net Estimation of (Near-) Low-Rank Drift Matrices in Ornstein-Uhlenbeck Processes}
\author{Dmytro Marushkevych\thanks{School of Mathematics and Statistics, UNSW Sydney. Email: \texttt{d.marushkevych@unsw.edu.au}} \and Francisco Pina\thanks{Department of Mathematics, University of Luxembourg. Email: \texttt{francisco.pina@uni.lu}} \and Mark Podolskij\thanks{Department of Mathematics, University of Luxembourg. Email: \texttt{mark.podolskij@uni.lu}}}
\date{\today}

\begin{document}
\maketitle

\begin{abstract}
We study estimation of the drift matrix in a continuously observed
high-dimensional Ornstein-Uhlenbeck process when the drift is exactly or
approximately low rank. In this setting, exact low rank induces non-stable directions and hence a non-ergodic regime, resulting in a poorly conditioned empirical covariance matrix. To address this difficulty, we introduce a Weighted
Nuclear Elastic Net Estimator that combines ridge regularization with a
nuclear-norm penalty expressed in the empirical likelihood geometry.

Under a general diagonalizable spectral framework, we establish
 oracle inequalities relative to arbitrary low-rank comparison
matrices. For near low-rank drifts, the approximation error is
naturally measured through the singular-value decay of the drift after
weighting by the regularized empirical covariance. The stochastic term is
controlled by self-normalized martingale arguments under appropriate choice of the tuning parameter. 

For a symmetric positive-semidefinite exact low-rank model, we verify the empirical-curvature condition required to
translate the weighted bound into a Frobenius-norm bound. With an appropriate
choice of tuning parameters, the resulting estimator satisfies, up to a
logarithmic factor, the standard rank-$r$ matrix-estimation scaling
$r d/T$: specifically, its squared Frobenius error is of order
$r d\log(T)/T$ with high probability, under an explicit
dimension-horizon condition.
\end{abstract}



\tableofcontents

\section{Introduction}\label{sec:introduction}

The multivariate Ornstein-Uhlenbeck (OU) process is one of the basic models for continuous-time stochastic dynamics.  In this paper, we consider a $d$-dimensional process $(X_t)_{t\geq 0}$ satisfying
\begin{equation}\label{eq:intro-ou}
    dX_t=-A_0X_t\,dt+D\,dW_t,
\end{equation}
observed continuously over the time interval $[0,T]$,
where $A_0\in\mathbb{R}^{d\times d}$ is an unknown drift matrix,  and $W$ is a standard $d$-dimensional Brownian motion. Due to complete path observation, $DD^{\top}$ can be assumed to be known. 
The model combines linear dependence, mean reversion, and a tractable Gaussian transition structure.  It originated in the physical description of Brownian motion \citep{UhlenbeckOrnstein1930} and has subsequently become a standard building block in, among other areas, interest-rate modelling \citep{Vasicek1977}, econometrics, biology, and engineering; see, for example, \citet{Gardiner2009}.

Equation~\eqref{eq:intro-ou} is the continuous-time analogue of a vector autoregression: sampling the process at a fixed interval $h>0$ yields a VAR(1) model with transition matrix $\exp(-hA_0)$. This connection, together with the explicit likelihood available under continuous observation, has made the OU process a fundamental benchmark for statistical inference in stochastic differential equations.
To briefly recall the classical fixed-dimensional theory, consider the diagonal case
\[
A_0=\operatorname{diag}(a_1,\ldots,a_d).
\]
If $a_i>0$ for all $i$, the process is ergodic and admits a unique stationary solution; we refer to corresponding eigenvectors as \textit{stable directions}. In this case, the maximum likelihood estimator (MLE) achieves the classical $\sqrt{T}$ convergence rate. In contrast, if some eigenvalues satisfy $a_i=0$ or $a_i<0$, the process is no longer ergodic;
the corresponding eigenvectors are referred to as \textit{non-stable directions}. The asymptotic behaviour of the MLE then depends on the spectral structure of $A_0$: directions corresponding to $a_i=0$ are estimated at rate $T$, whereas directions with $a_i<0$ are estimated at an exponential rate. We refer to \citet{Kutoyants2004} and \citet{BasakLee2008} for comprehensive treatments.

The high-dimensional regime is considerably more delicate: the $d^2$ entries of $A_0$ can be comparable with or substantially exceed the information available from a trajectory observed on $[0,T]$, so meaningful estimation requires additional structure.
Most existing high-dimensional theory for OU models has imposed an entrywise or row-wise sparsity condition on the drift matrix.  Under continuous observation and an ergodic OU model, \citet{GaiffasMatulewicz2019} developed Lasso and adaptive-Lasso procedures, obtaining non-asymptotic oracle inequalities together with asymptotic variable-selection results.  \citet{CiolekMarushkevychPodolskij2020} established oracle inequalities for Lasso and Dantzig estimators, improved the available rates, and showed that the relevant restricted-eigenvalue property follows from ergodicity rather than being imposed as an additional assumption.  Their later work \citep{CiolekMarushkevychPodolskij2025} extends Lasso theory to general high-dimensional parametric diffusion models, while \citet{MarushkevychPinaPodolskij2025} study support recovery and asymptotic normality for an adaptive-Lasso estimator.  These contributions provide a detailed understanding of sparse drift recovery from continuously observed diffusion trajectories.

Recent work has also broadened the observation schemes and driving-noise assumptions.  \citet{AmorinoPinaPodolskij2025} analyse the effect of discrete sampling on Lasso drift estimation and show when the continuous-observation rate can be retained.  For high-dimensional L\'{e}vy-driven OU processes, \citet{DexheimerJezska2026} derive sharp sparse oracle inequalities for Lasso and Slope estimators.  The common feature of this literature is that structural simplicity is encoded through zeros of $A_0$.  Such assumptions are well suited to network reconstruction, but they do not capture systems whose dynamics are governed by a small number of latent adjustment directions while individual matrix entries need not be sparse.

Low rank provides a different and genuinely global form of dimension reduction.  In general matrix regression, a rank constraint asserts that the rows and columns of the coefficient matrix lie in low-dimensional unknown subspaces.  Estimation with a direct rank penalty has been studied, for example, by \citet{BuneaSheWegkamp2011} and \citet{Klopp2011}; although statistically attractive, it leads to a non-convex optimisation problem.  The nuclear norm, defined as the sum of singular values, is the standard convex surrogate for rank and has become a central tool for low-rank matrix estimation \citep{RechtFazelParrilo2010,NegahbanWainwright2011,Koltchinskii2011}.  In particular, the framework of \citet{NegahbanWainwright2011} yields non-asymptotic Frobenius-norm bounds for both exactly and approximately low-rank matrices under restricted strong convexity.  Weighted or adaptive nuclear penalties can further reduce shrinkage bias and improve rank recovery in multivariate regression \citep{ChenDongChan2013}.

Beyond linear regression, reduced-rank drift matrices arise naturally in models for multivariate stochastic dynamics. One of the best-known examples is cointegration, where only a small number of long-run equilibrium relations govern the joint evolution of the system.  In a discrete-time vector error-correction model, the long-run matrix $\Pi=\alpha\beta^{\top}$ has reduced rank, and $\operatorname{rank}(\Pi)$ is the cointegration rank \citep{EngleGranger1987,Johansen1995}.  A related continuous-time theory has been developed for multivariate Gaussian OU diffusions by \citet{KesslerRahbek2004} and for broader multivariate OU models by \citet{Fasen2013}.  

Despite this connection, direct high-dimensional estimation of a low-rank OU drift remains comparatively underdeveloped.  \citet{BasakLee2008} treat estimation in general, possibly non-stable, multivariate OU systems, including configurations with zero eigenvalues, but do not consider low-rank regularisation or high-dimensional risk bounds.  In a more recent and different setting, \citet{Palaisti2026} studies a low-rank-plus-sparse drift for discretely observed L\'{e}vy-driven OU processes.  To the best of our knowledge, a non-asymptotic theory for nuclear-norm-based estimation of an exactly or approximately low-rank drift in a continuously observed Gaussian OU model has not previously been developed.

One reason why this problem has remained comparatively unexplored  is that, unlike in classical regression models, low-rank drift matrices in OU processes are not merely a structural assumption on the unknown parameter. Rather, they fundamentally change the probabilistic structure of the observations. Indeed, an exact low-rank drift necessarily introduces zero eigenvalues, so that the diffusion contains both stable and non-stable directions. As a consequence, the process ceases to be ergodic, moving  beyond the classical statistical framework for OU models. This change has important implications for statistical estimation, the first of which concerns the empirical design matrix
\[
C_T:=\frac{1}{T}\int_0^T X_tX_t^{\top}dt.
\]
 In the rank-deficient regime, the state process contains directions with fundamentally different growth behaviour, which may result in a highly heterogeneous spectrum of $C_T$. Under the natural normalization, $C_T$ can even become nearly singular; Section~\ref{sec:exact-low-rank} makes this phenomenon explicit in the exact low-rank setting. This intrinsic ill-conditioning complicates both the computation of the estimator and the non-asymptotic analysis of the likelihood contrast. Consequently, the estimation problem combines the challenges of high-dimensional statistics with those arising from a fundamentally different probabilistic regime.

In view of these challenges, the literature on regularized estimation for reduced-rank cointegrated systems provides useful methodological guidance, while also highlighting the limitations of existing approaches.  \citet{LevakovaDitlevsen2024} review rank, nuclear, adaptive-nuclear, Lasso, ridge, and elastic-net approaches for high-dimensional vector error-correction models.  Their numerical comparison shows that performance is strongly regime dependent: the ordinary nuclear penalty is a useful convex rank surrogate, but it is not uniformly the best-performing method and can overestimate the rank.  Nevertheless, their review highlights that rank reduction, often supplemented by additional regularisation, is essential in ill-conditioned cointegrated regressions.  In particular, they recommend adding a ridge term to nuclear-norm fitting to improve numerical stability.  This is consistent with the classical role of ridge regularisation under collinearity \citep{HoerlKennard1970} and with elastic-net regularisation more generally \citep{ZouHastie2005,MukherjeeZhu2011}.

Motivated by this feature, we introduce a \emph{Weighted Nuclear Elastic Net Estimator} (WNEE).  The estimator combines a nuclear-norm component, which promotes a low-rank drift, with an additional $\ell_2$-type stabilisation component.  The latter regularises the weakly identified directions induced by the random covariance structure, while the weighting is chosen to retain a rank-adaptive treatment of the singular spectrum. The ridge component is therefore not merely a numerical stabilisation device, but an essential ingredient for the theoretical analysis developed in this paper. The precise estimator and its tuning parameters are introduced in Section \ref{sec:model-estimator}.


The present article develops a non-asymptotic theory for estimation of low-rank and approximately low-rank drift matrices in multivariate OU processes. First, we formulate a convex, ridge-stabilised nuclear-norm procedure
tailored to the continuous-time likelihood. Second, we establish a high-probability oracle inequality in the empirical Fisher geometry of the observed diffusion trajectory, relative to arbitrary low-rank comparison
matrices. This result applies to both exactly and approximately low-rank targets. Third, under an empirical-curvature condition, we convert the weighted oracle inequality into a Frobenius-norm bound. Finally, we verify this curvature condition explicitly in a symmetric exact low-rank model and obtain an $rd\log(T)/T$ high-probability squared Frobenius-norm bound, where $r$ denotes the rank of $A_0$.

Section \ref{sec:model-estimator} introduces the model, the Weighted Nuclear Elastic Net Estimator, and the standing assumptions.  Section \ref{sec:main-results} states the main oracle inequalities.  Section \ref{sec:auxiliary-results} contains the auxiliary probabilistic results and the proofs of the main theorems.  Section \ref{sec:exact-low-rank} develops the exact low-rank regime, deriving refined score control, verifying the empirical-curvature condition and establishing explicit Frobenius-norm error bounds.  Section \ref{sec:numerical-study} presents a numerical study.

\section{Model, Assumptions and the Weighted Nuclear Elastic Net Estimator}
\label{sec:model-estimator}

\subsection{Notation}
\label{subsec:notation}

All random variables are defined on a filtered probability space $(\Omega, \mathcal{F}, (\mathcal{F}_t)_{t\geq 0},\mathbb{P})$. If we want to emphasize the dependence of 
probability measure or expectation on the underlying parameter
$\theta$, we write $\mathbb{P}_\theta$ or $\E_\theta$.

All vectors are column vectors. For $m,n\in\mathbb N$, let
$\mathbb M_{m\times n}:=\mathbb R^{m\times n}$ and write
$\mathbb M_d:=\mathbb M_{d\times d}$. For matrices $A,B$ of compatible
sizes, $A^\top$ denotes the transpose of $A$, $\operatorname{tr}(A)$ its trace,
and
\[
    \langle A,B\rangle_F:=\operatorname{tr}(A^\top B)
\]
denotes the Frobenius inner product. We write $\|A\|_F$,
$\|A\|_{\mathrm{op}}$, and $\|A\|_*$ for the Frobenius, operator, and nuclear
(Schatten-$1$) norms, respectively. If $q=m\wedge n$ and
$A\in\mathbb M_{m\times n}$, its singular values are denoted by
\[
    \sigma_1(A)\geq\sigma_2(A)\geq\cdots\geq\sigma_q(A)\geq 0,
\]
and $\operatorname{rank}(A)$ denotes its rank. Given a singular-value
decomposition $A=\sum_{j=1}^q \sigma_j(A)u_jv_j^\top$, define the rank-$s$
truncation
\[
    A_{[s]}:=\sum_{j=1}^s \sigma_j(A)u_jv_j^\top,
    \qquad s\in\{0,\ldots,q\},
\]
with $A_{[0]}:=0$.

For a symmetric matrix $H$, $\lambda_{\min}(H)$ and
$\lambda_{\max}(H)$ denote its smallest and largest eigenvalues. The notation
$H\succeq0$ means that $H$ is positive semidefinite. We write $I_d$ for the
$d\times d$ identity matrix, $(x)_+:=\max\{x,0\}$ for $x\in\mathbb R$, and
$a\wedge b:=\min\{a,b\}$ and $a\vee b:=\max\{a,b\}$. For vector-valued
processes $(U_t)$ and $(V_t)$, we use the empirical $L^2$ inner product
\[
    \langle U,V\rangle_{L^2_T}
    :=\frac{1}{T}\int_0^T U_t^\top V_t\,dt.
\]

\subsection{Continuous-time model and likelihood}
\label{subsec:model-likelihood}

Under continuous observation, the covariance matrix $Q=DD^\top$ in \eqref{eq:intro-ou}  is identified pathwise from quadratic variation. Hence, throughout the paper,
we consider continuous observations on $[0,T]$ of a $d$-dimensional Ornstein-Uhlenbeck
process satisfying
\begin{equation}
    dX_t=-A_0X_t\,dt+dW_t,
    \qquad 0\leq t\leq T,
\label{eq:ou-model}
\end{equation}
where $A_0\in\mathbb M_d$ is the unknown drift matrix, $(W_t)_{t\geq0}$ is a
standard $d$-dimensional Brownian motion, and $X_0=0$. The dimension $d=d_T$ is allowed to grow with the observation horizon $T$. Throughout Sections~2-5, we work with the normalized diffusion coefficient
in \eqref{eq:ou-model}.  The framework can also accommodate non-isotropic diffusion by imposing the structural assumptions on the whitened drift matrix.


Let $\mathbb P_A^T$ denote the law on $C([0,T];\mathbb R^d)$ of the model
\eqref{eq:ou-model} with drift parameter $A$. By Girsanov's theorem, the
log-likelihood ratio with respect to the driftless law $\mathbb P_0^T$ is
\[
    \log\frac{d\mathbb P_A^T}{d\mathbb P_0^T}(X)
    =-
    \int_0^T\langle AX_t,dX_t\rangle
    -\frac12\int_0^T\|AX_t\|_2^2\,dt.
\]
Accordingly, minimizing the normalized negative log-likelihood is equivalent to
minimizing the contrast
\begin{equation}
    \mathcal L_T(A)
    :=\frac{1}{T}\int_0^T\langle AX_t,dX_t\rangle
      +\frac{1}{2T}\int_0^T\|AX_t\|_2^2\,dt,
    \qquad A\in\mathbb M_d.
\label{eq:contrast}
\end{equation}
We
introduce the empirical covariance, the empirical score matrix, and the
martingale term via
\begin{equation}
    C_T:=\frac{1}{T}\int_0^T X_tX_t^\top\,dt,
    \qquad
    Z_T:=-\frac{1}{T}\int_0^T dX_tX_t^\top,
    \qquad
    \varepsilon_T:=\frac{1}{T}\int_0^T dW_tX_t^\top.
\label{eq:empirical-objects}
\end{equation}
The contrast then has the quadratic representation
\begin{equation}
    \mathcal L_T(A)
    =\frac12\|AC_T^{1/2}\|_F^2-\langle A,Z_T\rangle_F,
\label{eq:quadratic-likelihood}
\end{equation}
and the model identity in \eqref{eq:ou-model} gives
\begin{equation}
    Z_T=A_0C_T-\varepsilon_T,
    \qquad
    \nabla\mathcal L_T(A)=(A-A_0)C_T+\varepsilon_T.
\label{eq:score-decomposition}
\end{equation}
In particular, for every $A\in\mathbb M_d$,
\begin{equation}
    \mathcal L_T(A)-\mathcal L_T(A_0)
    =\frac12\big\|(A-A_0)C_T^{1/2}\big\|_F^2
      +\langle\varepsilon_T,A-A_0\rangle_F.
\label{eq:basic-likelihood-decomposition}
\end{equation}
Equation \eqref{eq:basic-likelihood-decomposition} separates the deterministic
quadratic curvature from the stochastic score term. When $C_T$ is invertible,
the unpenalized maximum-likelihood estimator is
\[
    \widehat A_{\mathrm{ML}}=Z_TC_T^{-1}.
\]
In the low-rank and high-dimensional regimes considered here, however, the
spectrum of $C_T$ may be markedly heterogeneous, making direct inversion poorly
conditioned even when $C_T$ is nonsingular.

\subsection{Weighted Nuclear Elastic Net Estimator}
\label{subsec:wnee}

For $\eta>0$, define the data-adaptive weight matrix
\begin{equation}
    B_{T,\eta}:=(C_T+\eta I_d)^{1/2}.
\label{eq:weight-matrix}
\end{equation}
Let $\lambda>0$ and $\eta>0$ be tuning parameters. The \emph{Weighted Nuclear
Elastic Net Estimator} (WNEE) is any minimizer
\begin{equation}
    \widehat A_{\lambda,\eta}
    \in\underset{A\in\mathbb M_d}{\operatorname{argmin}}
    \left\{
        \mathcal L_T(A)
        +\frac{\eta}{2}\|A\|_F^2
        +\lambda\big\|AB_{T,\eta}\big\|_*
    \right\}.
\label{eq:wnee}
\end{equation}
The second term is the ridge component, while the final term is an $\ell_1$
penalty on the singular values of $AB_{T,\eta}$.

Using \eqref{eq:quadratic-likelihood} and \eqref{eq:weight-matrix}, the smooth
part of the criterion in \eqref{eq:wnee} can be written as
\begin{equation}
    \mathcal L_T(A)+\frac{\eta}{2}\|A\|_F^2
    =\frac12\|AB_{T,\eta}\|_F^2-\langle A,Z_T\rangle_F.
\label{eq:weighted-geometry}
\end{equation}
Thus, the likelihood curvature and the nuclear penalty are expressed in the same
empirical geometry. The weight $B_{T,\eta}$ accounts for the fact that the
empirical information $C_T$ may scale differently in stable and non-stable
directions. The ridge component both regularizes weakly observed directions and
ensures that $B_{T,\eta}$ is invertible.

\begin{remark}
\label{rem:ridge-weight}
The ridge and weighted nuclear components play distinct roles. The ridge term
prevents small empirical eigenvalues of $C_T$ from destabilizing the optimization,
while the right weight $B_{T,\eta}$ calibrates singular-value shrinkage to the
information contained in the observed trajectory. Hence, the penalty is not
applied in a fixed Euclidean geometry: directions with different empirical scales
are regularized in comparable likelihood units.
\end{remark}

The non-asymptotic oracle inequalities are stated in
Section~\ref{sec:main-results}. Section~\ref{sec:auxiliary-results} gives both
an observable score-based choice of the tuning parameter $\lambda$ and, under
Assumption~\ref{ass:spectral-structure}, a deterministic alternative.

\subsection{Spectral framework}
\label{subsec:spectral-framework}

\begin{assumption}
\label{ass:spectral-structure}
There exist finite constants $\overline a>0$ and $\overline\kappa\geq1$, not
depending on $d$ or $T$, such that
\[
    A_0=P_0\Lambda_0P_0^{-1},
    \qquad
    \Lambda_0=\operatorname{diag}(\theta_1,\ldots,\theta_d),
\]
where
\[
    0\leq\theta_j\leq\overline a,
    \qquad j=1,\ldots,d,
    \qquad\text{and}\qquad
    \|P_0\|_{\mathrm{op}}\|P_0^{-1}\|_{\mathrm{op}}
    \leq\overline\kappa.
\]
\end{assumption}

\noindent
Assumption~\ref{ass:spectral-structure} allows zero eigenvalues and therefore
includes non-stable directions. It also excludes nontrivial Jordan blocks at zero,
which would lead to higher-order polynomial trends. The assumption does not
impose symmetry, positive semidefiniteness in the Euclidean geometry, or a known
rank on $A_0$.


The results in Sections~\ref{sec:model-estimator}-\ref{sec:auxiliary-results}
are formulated under the general spectral framework above and allow $A_0$ to be
only approximately low rank. In Section~\ref{sec:exact-low-rank}, we specialize
to the exact low-rank class
\[
    A_0=P\operatorname{diag}(a_1,\ldots,a_r,0,\ldots,0)P^\top,
\]
where $P$ is orthogonal and the nonzero eigenvalues are bounded away from zero
and infinity. This stronger structure yields an orthogonal decomposition into
stable Ornstein-Uhlenbeck and Brownian components. It is used in
Section ~\ref{sec:exact-low-rank} to obtain
explicit probability bounds for the score event and the empirical-curvature
condition.

\subsection{Empirical curvature condition}
\label{subsec:empirical-curvature}

The ridge-regularized contrast has the quadratic curvature
\begin{equation}
    \nabla^2\left(
        \mathcal L_T(\cdot)+\frac{\eta}{2}\|\cdot\|_F^2
    \right)(A)[M,M]
    =\|MB_{T,\eta}\|_F^2,
    \qquad A,M\in\mathbb M_d.
\label{eq:regularised-hessian}
\end{equation}
Hence, the following condition ensures that the ridge-regularized likelihood has
uniformly nondegenerate empirical curvature.

\begin{assumption}
\label{ass:empirical-curvature}
There exist a constant $c_0>0$, independent of $d$ and $T$, and a number
$\pi_T\in[0,1]$ such that the event
\begin{equation}
    \mathcal R_T(c_0)
    :=
    \left\{
        \|MB_{T,\eta}\|_F^2\geq c_0\|M\|_F^2
        \quad\text{for every }M\in\mathbb M_d
    \right\}
\label{eq:empirical-curvature-event}
\end{equation}
satisfies
\begin{equation}
    \mathbb P_{A_0}\big(\mathcal R_T(c_0)\big)\geq1-\pi_T.
\label{eq:empirical-curvature-probability}
\end{equation}
\end{assumption}

\noindent
Since $B_{T,\eta}^2=C_T+\eta I_d$, the event
$\mathcal R_T(c_0)$ is equivalent to
\[
    \lambda_{\min}(C_T+\eta I_d)\geq c_0.
\]
The deterministic inequality $C_T+\eta I_d\succeq\eta I_d$ always implies
\[
    \|MB_{T,\eta}\|_F^2\geq\eta\|M\|_F^2.
\]
However, this bound becomes uninformative when $\eta=\eta_T$ tends to zero.
Assumption~\ref{ass:empirical-curvature} requires a lower bound that remains
nondegenerate at the scale relevant for the oracle inequalities. 


In the general near-low-rank framework, Assumption~\ref{ass:empirical-curvature} is retained as a high-probability design condition. Section~\ref{sec:exact-low-rank} treats the symmetric exact low-rank model of Assumption~\ref{ass:empirical-curvature}, where the stable and Brownian directions admit an explicit orthogonal decomposition. This structure permits a direct non-asymptotic verification of the curvature condition: for every admissible choice of $c_0$, Corollary ~\ref{cor:exact-low-rank-curvature-simple} gives an explicit dimension-horizon condition under which \eqref{eq:empirical-curvature-probability} holds with a prescribed failure probability $\pi_T$.

\section{Main Results: Non-Asymptotic Oracle Inequalities}
\label{sec:main-results}

Let's recall that for $\eta>0$ 
\[
    B_{T,\eta}=(C_T+\eta I_d)^{1/2}
\]
and define the score event
\begin{equation}
    \mathcal S_T(\lambda,\eta)
    :=
    \left\{
        2\big\|\varepsilon_TB_{T,\eta}^{-1}\big\|_{\mathrm{op}}
        \leq\lambda
    \right\}.
\label{eq:score-event}
\end{equation}
For $r\in\{0,\ldots,d\}$, write
\[
    \mathcal M_r
    :=\big\{A\in\mathbb M_d:\operatorname{rank}(A)\leq r\big\}.
\]
The relevant approximation error is measured in the empirical likelihood
geometry. Accordingly, define
\begin{equation}
    \mathfrak a_{r,T,\eta}(A_0)
    :=
    \inf_{A\in\mathcal M_r}
    \big\|(A-A_0)B_{T,\eta}\big\|_F^2.
\label{eq:weighted-approximation-error}
\end{equation}
The infimum in
\eqref{eq:weighted-approximation-error} is attained at
\begin{equation}
    A_{0,[r]}^{T,\eta}
    :=
    \big(A_0B_{T,\eta}\big)_{[r]}B_{T,\eta}^{-1},
\label{eq:weighted-rank-r-approximation}
\end{equation}
where $(A_0B_{T,\eta})_{[r]}$ denotes the ordinary rank-$r$ singular-value
truncation of $A_0B_{T,\eta}$. 
Indeed, since \(B_{T,\eta}\) is invertible, the map
\[
A \longmapsto C:=AB_{T,\eta}
\]
is a bijection on \(\mathbb{R}^{d\times d}\) and preserves rank:
\begin{equation}\label{eq:rank_preservation_pr}
\operatorname{rank}(C)
=
\operatorname{rank}(AB_{T,\eta})
=
\operatorname{rank}(A).
\end{equation}
Therefore,
\[
\begin{aligned}
\mathfrak a_{r,T,\eta}(A_0)
&=
\inf_{A\in\mathcal M_r}
\left\|(A-A_0)B_{T,\eta}\right\|_F^2 \\
&=
\inf_{\operatorname{rank}(C)\le r}
\left\|C-A_0B_{T,\eta}\right\|_F^2.
\end{aligned}
\]
By the Eckart-Young-Mirsky theorem, one minimizer of the latter
problem is the rank-\(r\) singular-value truncation
\[
C^\star=(A_0B_{T,\eta})_{[r]}.
\]
Transforming back through \(A=CB_{T,\eta}^{-1}\), we obtain ~\eqref{eq:weighted-rank-r-approximation}.
Consequently, 
\begin{equation}
    \mathfrak a_{r,T,\eta}(A_0)
    =
    \sum_{j=r+1}^d\sigma_j^2\big(A_0B_{T,\eta}\big).
\label{eq:weighted-tail-singular-values}
\end{equation}
In particular, $\mathfrak a_{r,T,\eta}(A_0)=0$ whenever
$\operatorname{rank}(A_0)\leq r$.

For comparison with the usual unweighted notion of near low rank, let
$A_{0,[r]}$ denote the ordinary rank-$r$ singular-value truncation of $A_0$.
Then
\begin{equation}
    \mathfrak a_{r,T,\eta}(A_0)
    \leq
    \big\|(A_{0,[r]}-A_0)B_{T,\eta}\big\|_F^2
    \leq
    \|B_{T,\eta}\|_{\mathrm{op}}^2
    \sum_{j=r+1}^d\sigma_j^2(A_0).
\label{eq:unweighted-to-weighted-approximation}
\end{equation}
Thus, ordinary near low-rank structure provides a sufficient, although
potentially conservative, upper bound for the approximation error in the
empirical likelihood geometry. 

\begin{remark}
\label{rem:weighted-near-low-rank}
According to the exact-rank preservation property \eqref{eq:rank_preservation_pr}, if $A_0$ has exact rank at most $r$, then
$\mathfrak a_{r,T,\eta}(A_0)=0$, independently of the empirical covariance
matrix.

The situation is different for approximate low-rank structure. The oracle
inequality adapts directly to the singular-value decay of
$A_0B_{T,\eta}$, or equivalently to the weighted approximation error
$\mathfrak a_{r,T,\eta}(A_0)$, rather than to the unweighted singular-value
decay of $A_0$ alone. The conversion in
\eqref{eq:unweighted-to-weighted-approximation} is therefore only a sufficient
bound. Its factor $\|B_{T,\eta}\|_{\mathrm{op}}^2$ can be large when the
trajectory has directions with substantially different empirical scales, as may
occur in the presence of non-stable components.

This dependence on $B_{T,\eta}$ is intrinsic to the likelihood geometry. Directions with larger empirical energy contribute more
strongly to the likelihood contrast and are consequently required to be
approximated more accurately. Thus, for approximately low-rank models, the
natural structural quantity is the weighted tail
\[
    \sum_{j=r+1}^d
    \sigma_j^2\big(A_0B_{T,\eta}\big).
\]
\end{remark}

The next proposition assesses the probability of the set 
$\mathcal S_T(\lambda,\eta)$.
\begin{proposition}
\label{prop:deterministic-score-level}
Suppose that Assumption~\ref{ass:spectral-structure} holds. For every
$\delta\in(0,1)$, define
\begin{equation}
    \overline\lambda_{T,\eta,\delta}
    :=
    \frac{8\sqrt{2}}{3\sqrt{T}}
    \left\{
        d\log 9
        +\log\!\left(\frac{2}{\delta}\right)
        +\frac{d}{2}\log\!\left(
            1+\frac{\overline\kappa^2T}{\eta\delta}
        \right)
    \right\}^{1/2}.
\label{eq:deterministic-lambda}
\end{equation}
Then
\begin{equation}
    \mathbb P_{A_0}\!\left(
        \mathcal S_T(\overline\lambda_{T,\eta,\delta},\eta)
    \right)
    \geq 1-\delta.
\label{eq:deterministic-score-probability}
\end{equation}
\end{proposition}
The level in \eqref{eq:deterministic-lambda} is deterministic: it depends only
on $T$, $d$, $\eta$, $\delta$, and the spectral constant
$\overline\kappa$. In particular, if $\eta=\beta/T$ for a fixed $\beta>0$ and
$\delta_T=T^{-q}$ for some fixed $q>0$, then
\[
    \lambda_{T,\eta,\delta_T}^2
    \lesssim
    \frac{d_T\log T}{T}.
\]
The next theorem demonstrates the convergence rate of the estimator $\widehat A_{\lambda,\eta}$ with respect to Frobenius norm. 

\begin{theorem}
\label{thm:oracle-inequalities}
There exist universal numerical constants $C_1,C_2,C_3>0$ such that, on the
score event $\mathcal S_T(\lambda,\eta)$, the WNEE
$\widehat A_{\lambda,\eta}$ defined in \eqref{eq:wnee} satisfies for all $r\in\{0,\ldots,d\}$,
\begin{equation}
\begin{aligned}
    \big\|\big(\widehat A_{\lambda,\eta}-A_0\big)B_{T,\eta}\big\|_F^2
    \leq{}&
    C_1\mathfrak a_{r,T,\eta}(A_0)
    +C_2\lambda^2r \\
    &\quad
    +C_3\eta^2\big\|A_0B_{T,\eta}^{-1}\big\|_F^2.
\end{aligned}
\label{eq:weighted-oracle-inequality}
\end{equation}
Moreover, under the Assumption \ref{ass:empirical-curvature} on the event
$\mathcal S_T(\lambda,\eta)\cap\mathcal R_T(c_0)$,
\begin{equation}
\begin{aligned}
    \big\|\widehat A_{\lambda,\eta}-A_0\big\|_F^2
    \leq{}&
    \frac{C_1}{c_0}\mathfrak a_{r,T,\eta}(A_0)
    +\frac{C_2}{c_0}\lambda^2r 
    +\frac{C_3}{c_0}\eta^2
    \big\|A_0B_{T,\eta}^{-1}\big\|_F^2 \\
    \leq &
    \frac{C_1}{c_0}\mathfrak a_{r,T,\eta}(A_0)
    +\frac{C_2}{c_0}\lambda^2r+
    \frac{C_3}{c_0^2}\eta^2\|A_0\|_F^2.
\end{aligned}
\label{eq:frobenius-oracle-inequality}
\end{equation}
\end{theorem}


The three terms on the right-hand sides of \eqref{eq:weighted-oracle-inequality}
and \eqref{eq:frobenius-oracle-inequality} have distinct interpretations. For a
fixed rank $r$, the term $\mathfrak a_{r,T,\eta}(A_0)$ is the weighted
approximation error incurred by replacing $A_0$ with its best rank-$r$
approximation in the empirical likelihood geometry. It decreases as
$r$ increases and vanishes whenever $\operatorname{rank}(A_0)\le r$. 

The term $\lambda^2 r$ is the stochastic estimation cost of fitting a
rank-$r$ matrix. It is the principal statistical complexity term and
corresponds to the usual rank-adaptive matrix-estimation rate. 


Finally, the last term is the ridge-bias contribution. Indeed, after the change
of variables, the nuclear-norm procedure is centred at the
ridge-shrunken target
\[
A_{0,\eta}
=
A_0C_T(C_T+\eta I_d)^{-1},
\]
rather than at \(A_0\) itself. Although
\(\eta^2\|A_0B_{T,\eta}^{-1}\|_F^2\) depends on the observed empirical
covariance, it admits a deterministic upper bound on the curvature
event $\mathcal R_T(c_0)$, where 
\[
\eta^2\|A_0B_{T,\eta}^{-1}\|_F^2
\le
\frac{\eta^2}{c_0}\|A_0\|_F^2.
\]
Thus, in the Frobenius-norm inequality~\eqref{eq:frobenius-oracle-inequality},
the ridge-bias contribution is bounded by
\[
\frac{C_3}{c_0^2}\eta^2\|A_0\|_F^2,
\]
which no longer depends on the empirical
covariance matrix. In particular, if
\(\operatorname{rank}(A_0)\le r\) and
\(\|A_0\|_{\mathrm{op}}\le L\), then
\[
\frac{C_3}{c_0^2}\eta^2\|A_0\|_F^2
\le
\frac{C_3L^2}{c_0^2}\eta^2r.
\]

For approximately low-rank matrices, the oracle rank $r$ may be selected
implicitly by balancing the decreasing approximation term
$\mathfrak a_{r,T,\eta}(A_0)$ against the increasing complexity term
$\lambda^2r$. The desirable regime is one in which both the
approximation error and the ridge bias are no larger than this
complexity term; the resulting bound is then of order $\lambda^2r$.
In particular, in the exact low-rank case, taking
$r\ge \operatorname{rank}(A_0)$ eliminates the approximation term
entirely. With a sufficiently small ridge level, the ridge-bias term is
then of smaller order, so that $\lambda^2r$ becomes the leading term.

\begin{remark}
Equation~\eqref{eq:weighted-oracle-inequality} does not require the
empirical-curvature event and provides adaptation over all ranks in the
weighted empirical geometry. The Frobenius-norm conclusion in \eqref{eq:frobenius-oracle-inequality} requires the additional event $\mathcal R_T(c_0)$, whose probability is controlled in
Section \ref{sec:exact-low-rank} for the exact low-rank symmetric model. 
\end{remark}

\begin{corollary}
\label{cor:exact-low-rank-oracle}
Suppose that $\operatorname{rank}(A_0)\leq r$. Then, on the event
$\mathcal S_T(\lambda,\eta)\cap\mathcal R_T(c_0)$,
\begin{equation}
    \big\|\widehat A_{\lambda,\eta}-A_0\big\|_F^2
    \leq
    \frac{C_2}{c_0}\lambda^2r
    +\frac{C_3}{c_0}\eta^2
    \big\|A_0B_{T,\eta}^{-1}\big\|_F^2.
\label{eq:exact-low-rank-oracle-inequality}
\end{equation}
\end{corollary}

\noindent
Under Assumption~\ref{ass:spectral-structure}, taking
$\lambda=\overline\lambda_{T,\eta,\delta}$ from
Proposition~\ref{prop:deterministic-score-level} gives
\[
    \mathbb P_{A_0}\big(
        \mathcal S_T(\lambda,\eta)
    \big)\geq1-\delta.
\]
Together with Assumption~\ref{ass:empirical-curvature}, the Frobenius-norm
conclusion \eqref{eq:frobenius-oracle-inequality} therefore holds with
probability at least $1-\delta-\pi_T$. In the exact low-rank setting,
Corollary~\ref{cor:exact-low-rank-score-tuning} gives a sharper deterministic
score level, while Section~\ref{sec:exact-low-rank} provides an explicit bound
for $\pi_T$.

\section{Auxiliary Results and Proofs}
\label{sec:auxiliary-results}


Throughout this section, we write
\[
    B:=B_{T,\eta}=(C_T+\eta I_d)^{1/2},
    \qquad
    \overline C_T:=\int_0^T X_tX_t^\top\,dt=TC_T,
    \qquad
    \overline\varepsilon_T:=\int_0^T dW_tX_t^\top=T\varepsilon_T.
\]
Moreover, define the self-normalized score matrix
\begin{equation}
    \mathcal M_{T,\eta}
    :=\overline\varepsilon_T
      \big(\overline C_T+T\eta I_d\big)^{-1/2}
    =\sqrt T\,\varepsilon_TB^{-1}.
\label{eq:unnormalized-self-normalized-score}
\end{equation}
Consequently,
\begin{equation}
    \mathcal S_T(\lambda,\eta)
    =
    \left\{
        \|\mathcal M_{T,\eta}\|_{\mathrm{op}}
        \leq \frac{\sqrt T\,\lambda}{2}
    \right\}.
\label{eq:score-event-unnormalized}
\end{equation}

\subsection{Self-normalized control of the score}
\label{subsec:score-control}

The goal of this subsection is to prove Proposition \ref{prop:deterministic-score-level}.
We begin with a self-normalized moment inequality for a fixed left direction.
It is a continuous-time version of the pseudo-maximization argument used in
multivariate self-normalization; see \cite{DeLaPenaKlassLai2009}.

\begin{lemma}
\label{lem:self-normalized-moment}
Let $u\in\mathbb R^d$ satisfy $\|u\|_2=1$, and let $a>0$. Set
\[
    \xi_t(u):=\overline\varepsilon_t^\top u
    =\int_0^t X_s\,d\langle u,W_s\rangle,
    \qquad
    V_t:=\int_0^t X_sX_s^\top\,ds.
\]
Then
\begin{equation}
    \mathbb E_{A_0}\!\left[
        \exp\!\left(
            \frac14 \xi_T(u)^\top(V_T+aI_d)^{-1}\xi_T(u)
        \right)
    \right]
    \leq
    \left\{
        \mathbb E_{A_0}\!\left[
            \det\!\left(I_d+a^{-1}V_T\right)^{1/2}
        \right]
    \right\}^{1/2}.
\label{eq:self-normalized-moment}
\end{equation}
\end{lemma}

\begin{proof}
For every deterministic $\theta\in\mathbb R^d$, the exponential process
\[
    \mathcal E_t(\theta)
    :=\exp\!\left(
        \theta^\top \xi_t(u)-\frac12\theta^\top V_t\theta
    \right)
\]
is a nonnegative local martingale and hence a supermartingale. Integrating
$\mathcal E_t(\theta)$ with respect to the $N(0,a^{-1}I_d)$ density and
completing the square gives a nonnegative supermartingale
\[
    \mathcal Q_t
    :=
    \det\!\left(I_d+a^{-1}V_t\right)^{-1/2}
    \exp\!\left(
        \frac12 \xi_t(u)^\top(V_t+aI_d)^{-1}\xi_t(u)
    \right).
\]
In particular, $\mathbb E_{A_0}[\mathcal Q_T]\leq1$. By the Cauchy-Schwarz
inequality,
\begin{align*}
    \mathbb E_{A_0}\!\left[
        \exp\!\left(
            \frac14 \xi_T(u)^\top(V_T+aI_d)^{-1}\xi_T(u)
        \right)
    \right]
    &=
    \mathbb E_{A_0}\!\left[
        \mathcal Q_T^{1/2}
        \det\!\left(I_d+a^{-1}V_T\right)^{1/4}
    \right] \\
    &\leq
    \mathbb E_{A_0}[\mathcal Q_T]^{1/2}
    \left\{
        \mathbb E_{A_0}\!\left[
            \det\!\left(I_d+a^{-1}V_T\right)^{1/2}
        \right]
    \right\}^{1/2},
\end{align*}
which proves \eqref{eq:self-normalized-moment}.
\end{proof}

The next result represents a version of Proposition 
\ref{prop:deterministic-score-level} with a stochastic tuning parameter $\lambda_{T,\eta,\delta}$.
\begin{proposition}
\label{prop:data-driven-score-control}
For $\delta\in(0,1)$, define the observable tuning level
\begin{equation}
    \lambda_{T,\eta,\delta}
    :=
    \frac{8\sqrt{2}}{3\sqrt{T}}
    \left\{
        d\log 9+\log\!\left(\frac{1}{\delta}\right)
        +\frac12\log\det\!\left(I_d+\eta^{-1}C_T\right)
    \right\}^{1/2}.
\label{eq:data-driven-lambda}
\end{equation}
Then
\begin{equation}
    \mathbb P_{A_0}\!\left(
        \mathcal S_T(\lambda_{T,\eta,\delta},\eta)
    \right)
    \geq1-\delta.
\label{eq:data-driven-score-probability}
\end{equation}
\end{proposition}

\begin{proof}
Let $\mathcal N$ be a deterministic $1/4$-net of the Euclidean unit sphere in
$\mathbb R^d$ such that $|\mathcal N|\leq9^d$. For $u\in\mathcal N$, use the
mixture supermartingale constructed in the proof of
Lemma~\ref{lem:self-normalized-moment}. Markov's inequality gives
\begin{align*}
    \mathbb P_{A_0}\!\bigg(
        &\left\|\left(\overline C_T+T\eta I_d\right)^{-1/2}
        \overline\varepsilon_T^\top u\right\|_2^2 \\
        &\qquad>2\log\!\left[
            \frac{1}{\delta}
            \det\!\left(I_d+\eta^{-1}C_T\right)^{1/2}
        \right]
    \bigg)
    \leq\delta.
\end{align*}

Applying this bound with $\delta$ replaced by $\delta/9^d$ and taking a
union bound yields, with probability at least $1-\delta$,
\begin{equation}
    \max_{u\in\mathcal N}
    \left\|u^\top\mathcal M_{T,\eta}\right\|_2
    \leq \sqrt{2\ell_{T,\eta,\delta}},
\label{eq:net-data-driven-score-bound}
\end{equation}
where
\[
    \ell_{T,\eta,\delta}
    :=d\log 9+\log\!\left(\frac1\delta\right)
      +\frac12\log\det\!\left(I_d+\eta^{-1}C_T\right).
\]
For every matrix $H\in\mathbb R^{d\times d}$, the standard net inequality gives
\[
    \|H\|_{\mathrm{op}}
    \leq\frac{4}{3}\max_{u\in\mathcal N}\|u^\top H\|_2.
\]
Applying this to $M_{T,\eta}$ and using
\eqref{eq:net-data-driven-score-bound} shows that
\[
    2\|\varepsilon_TB^{-1}\|_{\mathrm{op}}
    =\frac{2}{\sqrt T}\|\mathcal M_{T,\eta}\|_{\mathrm{op}}
    \leq\lambda_{T,\eta,\delta},
\]
which proves \eqref{eq:data-driven-score-probability}.
\end{proof}

\begin{lemma}
\label{lem:spectral-trace-bound}
Under Assumption~\ref{ass:spectral-structure},
\begin{equation}
    \mathbb E_{A_0}\!\left[\operatorname{tr}(C_T)\right]
    \leq\frac{\overline\kappa^2dT}{2}.
\label{eq:spectral-trace-bound}
\end{equation}
\end{lemma}

\begin{proof}
Assumption~\ref{ass:spectral-structure} yields
\[
    \|\exp(-tA_0)\|_{\mathrm{op}}
    \leq
    \|P_0\|_{\mathrm{op}}
    \|\exp(-t\Lambda_0)\|_{\mathrm{op}}
    \|P_0^{-1}\|_{\mathrm{op}}
    \leq\overline\kappa,
    \qquad t\geq0.
\]
Since $X_0=0$,
\[
    X_t=\int_0^t \exp(-(t-s)A_0)\,dW_s.
\]
It\^o's isometry therefore gives
\begin{align*}
    \mathbb E_{A_0}\!\left[\operatorname{tr}(C_T)\right]
    &=\frac1T\int_0^T
      \int_0^t
      \|\exp(-(t-s)A_0)\|_F^2\,ds\,dt \\
    &\leq\frac1T\int_0^T\int_0^t
      d\overline\kappa^2\,ds\,dt
    =\frac{\overline\kappa^2dT}{2}.
\end{align*}
\end{proof}

\begin{proof}[Proof of Proposition~\ref{prop:deterministic-score-level}]
Let
\[
    \mathcal D_{T,\delta}
    :=\left\{
        \operatorname{tr}(C_T)
        \leq\frac{\overline\kappa^2dT}{\delta}
    \right\}.
\]
By Lemma~\ref{lem:spectral-trace-bound} and Markov's inequality,
\[
    \mathbb P_{A_0}\!\left(\mathcal D_{T,\delta}^c\right)
    \leq\frac{\delta}{2}.
\]
For every positive-semidefinite matrix $H$, the arithmetic-geometric mean
inequality yields
\[
    \det(I_d+H)
    \leq\left(1+\frac{\operatorname{tr}(H)}{d}\right)^d.
\]
Hence, on $\mathcal D_{T,\delta}$,
\begin{equation}
    \log\det\!\left(I_d+\eta^{-1}C_T\right)
    \leq
    d\log\!\left(
        1+\frac{\overline\kappa^2T}{\eta\delta}
    \right).
\label{eq:deterministic-logdet-envelope}
\end{equation}
Combining \eqref{eq:data-driven-lambda} with \eqref{eq:deterministic-logdet-envelope},
we obtain, on $\mathcal D_{T,\delta}$,
\[
\lambda_{T,\eta,\delta/2}
\le \overline\lambda_{T,\eta,\delta}.
\]
Moreover, Proposition \ref{prop:data-driven-score-control}, applied with confidence level $\delta/2$,
gives
\[
\mathbb P_{A_0}\!\left(
\mathcal S_T(\lambda_{T,\eta,\delta/2},\eta)^c
\right)\le \frac{\delta}{2}.
\]
Consequently,
\begin{align*}
    \mathbb P_{A_0}\!\left(
        \mathcal S_T(\overline\lambda_{T,\eta,\delta},\eta)^c
    \right)
    &\leq
    \mathbb P_{A_0}\!\left(
        \mathcal S_T(\lambda_{T,\eta,\delta/2},\eta)^c
    \right)
    +\mathbb P_{A_0}\!\left(\mathcal D_{T,\delta}^c\right) \\
    &\leq\frac\delta2+\frac\delta2
    =\delta.
\end{align*}
\end{proof}


\subsection{A deterministic nuclear-norm inequality}
\label{subsec:deterministic-nuclear-inequality}

This and the subsequent subsections present auxiliary results, which are necessary to prove Theorem \ref{thm:oracle-inequalities}.
The next lemma is a deterministic oracle inequality for nuclear-norm regularized
matrix denoising. 

\begin{lemma}
\label{lem:nuclear-denoising}
Let $\Theta,G\in\mathbb M_d$, let $\lambda>0$, and define
\[
    \widehat\Theta
    \in\underset{U\in\mathbb M_d}{\operatorname{argmin}}
    \left\{
        \frac12\|U-(\Theta+G)\|_F^2+\lambda\|U\|_*
    \right\}.
\]
If $\|G\|_{\mathrm{op}}\leq\lambda/2$, then, for every
$H\in\mathbb M_d$ with $\operatorname{rank}(H)\leq s$,
\begin{equation}
    \|\widehat\Theta-\Theta\|_F^2
    \leq
    2\|H-\Theta\|_F^2+36\lambda^2s.
\label{eq:nuclear-denoising-oracle}
\end{equation}
\end{lemma}

\begin{proof}
Set $\Delta:=\widehat\Theta-\Theta$, $D:=\widehat\Theta-H$, and
$e:=H-\Theta$. The optimality of $\widehat\Theta$ gives
\begin{equation}
    \frac12\|\Delta\|_F^2
    \leq
    \frac12\|e\|_F^2
    +\langle G,D\rangle_F
    +\lambda\big(\|H\|_*-\|\widehat\Theta\|_*\big).
\label{eq:basic-nuclear-inequality}
\end{equation}
Let $H=U\Sigma V^\top$ be a compact singular-value decomposition, with
$k:=\operatorname{rank}(H)\leq s$, and define the tangent space
\[
    \mathcal T_H
    :=\big\{UK^\top+LV^\top:K,L\in\mathbb R^{d\times k}\big\}.
\]
Let
$P_{\mathcal T_H}$ and $P_{\mathcal T_H^\perp}$ denote the
Frobenius-orthogonal projections onto $\mathcal T_H$ and its orthogonal
complement, respectively. Since
\[
P_{\mathcal T_H^\perp}D
=
(I_d-UU^\top)D(I_d-VV^\top),
\]
the matrices $H=U\Sigma V^\top$ and
$P_{\mathcal T_H^\perp}D$ have mutually orthogonal row and column
spaces. Hence
\[
\|H+P_{\mathcal T_H^\perp}D\|_*
=
\|H\|_*+\|P_{\mathcal T_H^\perp}D\|_*.
\]
Using $\widehat\Theta=H+P_{\mathcal T_H}D+
P_{\mathcal T_H^\perp}D$ and the triangle inequality, we obtain

\begin{equation}
    \|H\|_*-\|\widehat\Theta\|_*
    \leq
    \|\mathcal P_{\mathcal T_H}D\|_*
    -\|\mathcal P_{\mathcal T_H^\perp}D\|_*.
\label{eq:nuclear-decomposability}
\end{equation}
Moreover, trace duality and the score condition imply
\[
    \langle G,D\rangle_F
    \leq\frac\lambda2\|D\|_*
    \leq\frac\lambda2
    \left(
        \|\mathcal P_{\mathcal T_H}D\|_*
        +\|\mathcal P_{\mathcal T_H^\perp}D\|_*
    \right).
\]
Substituting these bounds into \eqref{eq:basic-nuclear-inequality} gives
\[
    \frac12\|\Delta\|_F^2
    \leq
    \frac12\|e\|_F^2
    +\frac{3\lambda}{2}\|\mathcal P_{\mathcal T_H}D\|_*.
\]
Since $\operatorname{rank}(\mathcal P_{\mathcal T_H}D)\leq2s$,
\[
    \|\mathcal P_{\mathcal T_H}D\|_*
    \leq\sqrt{2s}\|D\|_F
    \leq\sqrt{2s}\big(\|\Delta\|_F+\|e\|_F\big).
\]
Let $x:=\|\Delta\|_F$, $b:=\|e\|_F$, and
$a:=3\lambda\sqrt{2s}$. Multiplying the preceding inequality by two yields
\[
    x^2\leq b^2+a(x+b).
\]
Using $ax\leq x^2/4+a^2$ and $ab\leq b^2/2+a^2/2$, we obtain
\[
    x^2\leq b^2+\frac{x^2}{4}+a^2+\frac{b^2}{2}+\frac{a^2}{2},
\]
and hence $x^2\leq2b^2+2a^2$. Since $2a^2=36\lambda^2s$, this proves
\eqref{eq:nuclear-denoising-oracle}.
\end{proof}

\subsection{Approximation after ridge regularization}
\label{subsec:ridge-approximation}

Define the ridge-shrunken target
\begin{equation}
    A_{0,\eta}:=A_0C_T(C_T+\eta I_d)^{-1}
    =A_0C_TB^{-2}.
\label{eq:ridge-shrunken-target}
\end{equation}

\begin{lemma}
\label{lem:ridge-approximation}
For every $A\in\mathbb M_d$, let
\[
    A_\eta:=AC_T(C_T+\eta I_d)^{-1}.
\]
Then $\operatorname{rank}(A_\eta)\leq\operatorname{rank}(A)$ and
\begin{equation}
    \big\|(A_\eta-A_{0,\eta})B\big\|_F
    \leq\big\|(A-A_0)B\big\|_F.
\label{eq:ridge-approximation-bound}
\end{equation}
Moreover,
\begin{equation}
    \big\|(A_{0,\eta}-A_0)B\big\|_F
    =\eta\big\|A_0B^{-1}\big\|_F.
\label{eq:ridge-bias-identity}
\end{equation}
\end{lemma}

\begin{proof}
Set
\[
    Q_{T,\eta}:=C_T(C_T+\eta I_d)^{-1}=I_d-\eta B^{-2}.
\]
Then $0\preceq Q_{T,\eta}\preceq I_d$, and $Q_{T,\eta}$ commutes with $B$.
Since $A_\eta=AQ_{T,\eta}$ and $A_{0,\eta}=A_0Q_{T,\eta}$,
\[
    (A_\eta-A_{0,\eta})B
    =(A-A_0)BQ_{T,\eta}.
\]
The operator norm of $Q_{T,\eta}$ is at most one, which proves
\eqref{eq:ridge-approximation-bound}. The rank assertion is immediate. Finally,
\[
    (A_{0,\eta}-A_0)B
    =A_0(Q_{T,\eta}-I_d)B
    =-\eta A_0B^{-1},
\]
which proves \eqref{eq:ridge-bias-identity}.
\end{proof}

\subsection{Proofs of the oracle inequalities}
\label{subsec:proofs-oracle}

\begin{proof}[Proof of Theorem~\ref{thm:oracle-inequalities}]
Fix $r\in\{0,\ldots,d\}$. Since $B$ is invertible, the minimizer in
\eqref{eq:weighted-approximation-error} is
\[
    A_r^\star:=(A_0B)_{[r]}B^{-1},
\]
so that $\operatorname{rank}(A_r^\star)\leq r$ and
\begin{equation}
    \big\|(A_r^\star-A_0)B\big\|_F^2
    =\mathfrak a_{r,T,\eta}(A_0).
\label{eq:best-weighted-approximator}
\end{equation}
Let
\[
    \widehat\Theta:=\widehat A_{\lambda,\eta}B,
    \qquad
    \Theta_{0,\eta}:=A_{0,\eta}B=A_0C_TB^{-1},
    \qquad
    G_{T,\eta}:=-\varepsilon_TB^{-1}.
\]
The identities in Section~\ref{sec:model-estimator} imply
\[
    Z_TB^{-1}=\Theta_{0,\eta}+G_{T,\eta}.
\]
Moreover, the change of variables $\Theta=AB$ transforms the criterion in
\eqref{eq:wnee}, up to an additive constant not depending on $\Theta$, into
\[
    \frac12\big\|\Theta-Z_TB^{-1}\big\|_F^2+\lambda\|\Theta\|_*.
\]
Thus, on $\mathcal S_T(\lambda,\eta)$,
$\|G_{T,\eta}\|_{\mathrm{op}}\leq\lambda/2$ and
Lemma~\ref{lem:nuclear-denoising} applies with
\[
    \Theta=\Theta_{0,\eta},
    \qquad
    H:=A_{r,\eta}^\star B,
    \qquad
    A_{r,\eta}^\star:=A_r^\star C_T(C_T+\eta I_d)^{-1}.
\]
By Lemma~\ref{lem:ridge-approximation} and
\eqref{eq:best-weighted-approximator},
\[
    \big\|H-\Theta_{0,\eta}\big\|_F^2
    \leq\mathfrak a_{r,T,\eta}(A_0).
\]
Hence,
\begin{equation}
    \big\|\widehat\Theta-\Theta_{0,\eta}\big\|_F^2
    \leq
    2\mathfrak a_{r,T,\eta}(A_0)+36\lambda^2r.
\label{eq:oracle-ridge-target}
\end{equation}
The ridge-bias identity in Lemma~\ref{lem:ridge-approximation} and the inequality
$\|U+V\|_F^2\leq2\|U\|_F^2+2\|V\|_F^2$ give
\begin{align*}
    \big\|\big(\widehat A_{\lambda,\eta}-A_0\big)B\big\|_F^2
    &\leq
    2\big\|\widehat\Theta-\Theta_{0,\eta}\big\|_F^2
    +2\big\|\big(A_{0,\eta}-A_0\big)B\big\|_F^2 \\
    &\leq
    4\mathfrak a_{r,T,\eta}(A_0)
    +72\lambda^2r
    +2\eta^2\big\|A_0B^{-1}\big\|_F^2.
\end{align*}
This proves \eqref{eq:weighted-oracle-inequality} with
$(C_1,C_2,C_3)=(4,72,2)$. On the event $\mathcal R_T(c_0)$,
\[
    c_0\big\|\widehat A_{\lambda,\eta}-A_0\big\|_F^2
    \leq
    \big\|\big(\widehat A_{\lambda,\eta}-A_0\big)B\big\|_F^2,
\]
and 
\[
\|A_0B_{T,\eta}^{-1}\|_F^2
\le c_0^{-1}\|A_0\|_F^2.
\]
which proves \eqref{eq:frobenius-oracle-inequality}.
\end{proof}

\begin{proof}[Proof of Corollary~\ref{cor:exact-low-rank-oracle}]
If $\operatorname{rank}(A_0)\leq r$, then
$A_0\in\mathcal M_r$ and consequently
$\mathfrak a_{r,T,\eta}(A_0)=0$. The result follows immediately from the
Frobenius-norm inequality in Theorem~\ref{thm:oracle-inequalities}.
\end{proof}


\section{Exact Low-Rank Results}
\label{sec:exact-low-rank}


This section develops explicit non-asymptotic results for the exactly low-rank symmetric Ornstein–Uhlenbeck model. In particular, we derive a refined deterministic score level, verify the empirical-curvature condition of Assumption \ref{ass:empirical-curvature}, and combine these results to obtain an explicit Frobenius-norm error bound for the WNEE.

\subsection{Exact low-rank regime and explicit bounds}
\label{subsec:exact-low-rank-regime}

\begin{assumption}
\label{ass:exact-low-rank}
Let $r=r_T$ and $m=m_T:=d_T-r_T$, where $1\leq r_T\leq d_T-1$. Suppose that
\eqref{eq:ou-model} holds with $X_0=0$ and
\begin{equation}
    A_0
    =P
    \begin{pmatrix}
        \operatorname{diag}(a_1,\ldots,a_r) & 0\\
        0 & 0_{m\times m}
    \end{pmatrix}
    P^\top,
\label{eq:exact-low-rank-drift}
\end{equation}
where $P\in\mathbb R^{d\times d}$ is orthogonal. There exist constants
\begin{equation}
    0<a_-\leq a_+<\infty
\label{eq:exact-low-rank-spectral-bounds}
\end{equation}
that do not depend on $d$ or $T$ such that
\[
    a_-\leq a_i\leq a_+,
    \qquad i=1,\ldots,r.
\]
\end{assumption}

\noindent
Assumption~\ref{ass:exact-low-rank} is a subclass of
Assumption~\ref{ass:spectral-structure}, with $\overline\kappa=1$ and
$\overline a=a_+$. It is more restrictive because it assumes that $A_0$ is
symmetric positive semidefinite and that its nonzero eigenvalues are uniformly
bounded away from zero. These additional properties yield an orthogonal
decomposition into stable Ornstein-Uhlenbeck and Brownian coordinates, which
is the basis of the non-asymptotic argument below.


Proposition \ref{prop:deterministic-score-level} provides a deterministic score calibration under the general
spectral framework of Assumption \ref{ass:spectral-structure}. That bound is obtained by controlling
the empirical covariance through its trace and therefore treats all coordinate
directions uniformly. Under Assumption \ref{ass:exact-low-rank}, the orthogonal decomposition into
stable Ornstein-Uhlenbeck and Brownian coordinates permits a more block-sensitive analysis.

The results below proceed in three steps. First, Proposition \ref{prop:general-score-control} bounds the
score tail probability in terms of a determinant moment of the empirical
covariance. Proposition \ref{prop:exact-low-rank-score-control} then controls this moment separately on the
stable and Brownian blocks, yielding a deterministic tuning level that reflects
their distinct empirical scales. Next, Theorem \ref{thm:exact-low-rank-curvature} verifies the empirical-curvature
condition by combining lower bounds for the two diagonal blocks with control of
the cross block. Finally, Corollary \ref{cor:exact-low-rank-rate} combines these ingredients with the
oracle inequality to obtain an explicit Frobenius-norm error bound.

\begin{proposition}
\label{prop:general-score-control}
For every $\eta>0$, $z>0$ and $\rho\in(0,1)$,
\begin{equation}
\begin{aligned}
    \mathbb P_{A_0}\!\left(
        \|\mathcal M_{T,\eta}\|_{\mathrm{op}}>z
    \right)
    \leq{}&
    \exp\!\left(
        -\frac{(1-\rho)^2z^2}{4}
        +d\log\!\left(1+\frac2\rho\right)
    \right) \\
    &\quad\times
    \left\{
        \mathbb E_{A_0}\!\left[
            \det\!\left(I_d+\eta^{-1}C_T\right)^{1/2}
        \right]
    \right\}^{1/2}.
\end{aligned}
\label{eq:general-operator-score-bound}
\end{equation}
Equivalently,
\begin{equation}
\begin{aligned}
    \mathbb P_{A_0}\!\left(
        \mathcal S_T(\lambda,\eta)^c
    \right)
    \leq{}&
    \exp\!\left(
        -\frac{(1-\rho)^2T\lambda^2}{16}
        +d\log\!\left(1+\frac2\rho\right)
    \right) \\
    &\quad\times
    \left\{
        \mathbb E_{A_0}\!\left[
            \det\!\left(I_d+\eta^{-1}C_T\right)^{1/2}
        \right]
    \right\}^{1/2}.
\end{aligned}
\label{eq:general-score-event-bound}
\end{equation}
\end{proposition}

\begin{proof}
Let $\mathcal N_\rho$ be a $\rho$-net of the Euclidean unit sphere
$\mathbb S^{d-1}$ with cardinality
$|\mathcal N_\rho|\leq(1+2/\rho)^d$. For $u\in\mathcal N_\rho$, apply
Lemma~\ref{lem:self-normalized-moment} with $a=T\eta$ and use Markov's
inequality to obtain, for every $x>0$,
\begin{align*}
    \mathbb P_{A_0}\!\left(
        \|u^\top\mathcal M_{T,\eta}\|_2>x
    \right)
    &\leq
    \exp\!\left(-\frac{x^2}{4}\right)
    \left\{
        \mathbb E_{A_0}\!\left[
            \det\!\left(
                I_d+(T\eta)^{-1}\overline C_T
            \right)^{1/2}
        \right]
    \right\}^{1/2} \\
    &=
    \exp\!\left(-\frac{x^2}{4}\right)
    \left\{
        \mathbb E_{A_0}\!\left[
            \det\!\left(I_d+\eta^{-1}C_T\right)^{1/2}
        \right]
    \right\}^{1/2}.
\end{align*}
For every $d\times d$ matrix $H$,
\[
    \|H\|_{\mathrm{op}}
    \leq \frac{1}{1-\rho}
    \max_{u\in\mathcal N_\rho}\|u^\top H\|_2.
\]
Applying this inequality to $H=\mathcal M_{T,\eta}$, followed by a union
bound over $\mathcal N_\rho$, proves
\eqref{eq:general-operator-score-bound}. Equation
\eqref{eq:general-score-event-bound} follows from
\eqref{eq:score-event-unnormalized}.
\end{proof}

\begin{proposition}
\label{prop:exact-low-rank-score-control}
Suppose that Assumption~\ref{ass:exact-low-rank} holds. Then, for every
$\eta>0$ and $\rho\in(0,1)$,
\begin{equation}
\begin{aligned}
    \mathbb P_{A_0}\!\left(
        \mathcal S_T(\lambda,\eta)^c
    \right)
    \leq{}&
    \exp\!\Bigg(
        -\frac{(1-\rho)^2T\lambda^2}{16}
        +d\log\!\left(1+\frac2\rho\right) \\
    &\qquad\quad
        +\frac r4\log\!\left(1+\frac{1}{2a_-\eta}\right)
        +\frac{d-r}{4}\log\!\left(1+\frac{T}{2\eta}\right)
    \Bigg).
\end{aligned}
\label{eq:sharpened-score-event-bound}
\end{equation}
\end{proposition}

\begin{proof}
An orthogonal change of coordinates preserves the operator norm in
\eqref{eq:score-event}. Under Assumption~\ref{ass:exact-low-rank}, we may
therefore assume without loss of generality that $P=I_d$. Then the coordinates
of $X$ are independent. For $i\leq r$, $X^i$ is an
Ornstein-Uhlenbeck coordinate started from zero with mean square
\[
    \mathbb E_{A_0}\!\left[(X_t^i)^2\right]
    =\frac{1-\exp(-2a_it)}{2a_i},
\]
whereas for $i>r$, $X^i$ is a Brownian motion. Consequently,
\begin{equation}
    \mathbb E_{A_0}[\overline C_T^{ii}]
    =\frac{T}{2a_i}
      -\frac{1-\exp(-2a_iT)}{4a_i^2}
    \leq\frac{T}{2a_-},
    \qquad i\leq r,
\label{eq:stable-diagonal-covariance-expectation}
\end{equation}
and
\begin{equation}
    \mathbb E_{A_0}[\overline C_T^{ii}]
    =\frac{T^2}{2},
    \qquad i>r.
\label{eq:brownian-diagonal-covariance-expectation}
\end{equation}
\noindent
Put $\beta:=T\eta$. By Hadamard's inequality, the independence of the diagonal
coordinates, and Jensen's inequality,
\begin{align*}
    \mathbb E_{A_0}\!\left[
        \det\!\left(I_d+\beta^{-1}\overline C_T\right)^{1/2}
    \right]
    &\leq
    \prod_{i=1}^d
    \mathbb E_{A_0}\!\left[
        \left(1+\beta^{-1}\overline C_T^{ii}\right)^{1/2}
    \right] \\
    &\leq
    \left(1+\frac{T}{2a_-\beta}\right)^{r/2}
    \left(1+\frac{T^2}{2\beta}\right)^{(d-r)/2}.
\end{align*}
Taking square roots, using $\beta=T\eta$, and substituting the resulting bound
into \eqref{eq:general-score-event-bound} proves
\eqref{eq:sharpened-score-event-bound}.
\end{proof}

\begin{corollary}
\label{cor:exact-low-rank-score-tuning}
Under Assumption~\ref{ass:exact-low-rank}, let $\eta>0$, $\delta\in(0,1)$ and
$\rho\in(0,1)$. If
\begin{equation}
\begin{aligned}
    \lambda^2
    \geq
    \frac{16}{(1-\rho)^2T}
    \Bigg[
        &\log\!\left(\frac1\delta\right)
        +d\log\!\left(1+\frac2\rho\right) \\
        &+\frac r4\log\!\left(1+\frac{1}{2a_-\eta}\right)
        +\frac{d-r}{4}\log\!\left(1+\frac{T}{2\eta}\right)
    \Bigg],
\end{aligned}
\label{eq:exact-low-rank-tuning-general-rho}
\end{equation}
then
\[
    \mathbb P_{A_0}\!\left(
        \mathcal S_T(\lambda,\eta)
    \right)
    \geq1-\delta.
\]
In particular, the explicit choice $\rho=1/2$ gives the sufficient condition
\begin{equation}
\begin{aligned}
    \lambda^2
    \geq
    \frac{64}{T}
    \Bigg[
        &\log\!\left(\frac1\delta\right)
        +d\log 5 \\
        &+\frac r4\log\!\left(1+\frac{1}{2a_-\eta}\right)
        +\frac{d-r}{4}\log\!\left(1+\frac{T}{2\eta}\right)
    \Bigg].
\end{aligned}
\label{eq:exact-low-rank-tuning-rho-half}
\end{equation}
\end{corollary}

\begin{proof}
Under the displayed condition, the right-hand side of inequality \eqref{eq:sharpened-score-event-bound} in 
Proposition \ref{prop:exact-low-rank-score-control} is bounded by $\delta$.
\end{proof}

For the curvature result, fix
\[
0<c_0<b:=\frac{1}{2a_+},
\]
and define

\begin{equation}
    \Delta:=b-c_0,
    \qquad
    c_1:=\frac{b+c_0}{2},
    \qquad
    h:=\sqrt{\frac{c_1}{c_0}}-1.
\label{eq:curvature-constants}
\end{equation}
The constants $C_S=C_S(a_-)>0$, $c_B>0$, and
$C_R:=8\log(9)/a_-^2$ are specified in
Lemmas~\ref{lem:exact-low-rank-stable-block}-\ref{lem:exact-low-rank-cross-block}.

\begin{theorem}
\label{thm:exact-low-rank-curvature}
Suppose that Assumption~\ref{ass:exact-low-rank} holds. Fix
$0<c_0<(2a_+)^{-1}$ and $\delta\in(0,1)$, and let
\[
    \ell_\delta:=\log\!\left(\frac{6}{\delta}\right).
\]
If $T\geq2$ and
\begin{align}
    T &\geq \frac{3}{2a_-^2\Delta},
\label{eq:curvature-init-condition}\\
    T &\geq \frac{144C_S^2}{\Delta^2}\bigl(r+\ell_\delta\bigr),
\label{eq:curvature-stable-square-condition}\\
    T &\geq \frac{12C_S}{\Delta}\bigl(r+\ell_\delta\bigr),
\label{eq:curvature-stable-linear-condition}\\
    T &\geq \frac{6C_R}{\Delta}\bigl(d+\ell_\delta\bigr),
\label{eq:curvature-cross-condition}\\
    T &\geq \frac{c_1}{c_B}\bigl(m+\ell_\delta\bigr),
\label{eq:curvature-brownian-condition}\\
    T &\geq
    \frac{1}{h}\sqrt{\frac{C_R}{c_B}}
    \sqrt{\bigl(d+\ell_\delta\bigr)\bigl(m+\ell_\delta\bigr)},
\label{eq:curvature-triangular-condition}
\end{align}
then, for every $\eta\geq0$,
\begin{equation}
    \mathbb P_{A_0}\bigl(\mathcal R_T(c_0)\bigr)
    \geq1-\delta.
\label{eq:exact-low-rank-curvature-probability}
\end{equation}
Equivalently,
\[
    \mathbb P_{A_0}\!\left(
        \lambda_{\min}(C_T+\eta I_d)\geq c_0
    \right)
    \geq1-\delta.
\]
\end{theorem}

\begin{corollary}
\label{cor:exact-low-rank-curvature-simple}
Under Assumption~\ref{ass:exact-low-rank}, fix
$0<c_0<(2a_+)^{-1}$. There is a constant
\[
    K=K(a_-,a_+,c_0)>0
\]
such that, for every $\delta\in(0,1)$ and $T\geq2$,
\begin{equation}
    T\geq K\left\{d+\log\!\left(\frac{6}{\delta}\right)\right\}
\label{eq:exact-low-rank-simple-horizon-condition}
\end{equation}
implies that
\[
    \mathbb P_{A_0}\bigl(\mathcal R_T(c_0)\bigr)\geq1-\delta
\]
for every deterministic $\eta\geq0$.
\end{corollary}

\begin{proof}
Since $r\leq d$, $m\leq d$, and
\[
    \sqrt{\bigl(d+\ell_\delta\bigr)\bigl(m+\ell_\delta\bigr)}
    \leq d+\ell_\delta,
\]
all six conditions in Theorem~\ref{thm:exact-low-rank-curvature} follow from
\eqref{eq:exact-low-rank-simple-horizon-condition} after increasing the
constant $K$.
\end{proof}

\begin{corollary}
\label{cor:exact-low-rank-rate}
Suppose that Assumption~\ref{ass:exact-low-rank} holds, fix
$0<c_0<(2a_+)^{-1}$ and $\beta>0$, and set $\eta=\beta/T$. Let
$\delta\in(0,1)$ and choose $\lambda$ so that
\begin{equation}
\begin{aligned}
    \lambda^2
    \geq
    \frac{64}{T}
    \Bigg[
        &\log\!\left(\frac{2}{\delta}\right)
        +d\log 5 \\
        &+\frac r4\log\!\left(1+\frac{T}{2a_-\beta}\right)
        +\frac{d-r}{4}\log\!\left(1+\frac{T^2}{2\beta}\right)
    \Bigg].
\end{aligned}
\label{eq:exact-low-rank-rate-tuning}
\end{equation}
If
\[
    T\geq K\left\{d+\log\!\left(\frac{12}{\delta}\right)\right\},
\]
where $K$ is the constant in
Corollary~\ref{cor:exact-low-rank-curvature-simple}, then, with probability at
least $1-\delta$,
\begin{equation}
    \big\|\widehat A_{\lambda,\eta}-A_0\big\|_F^2
    \leq
    \frac{72}{c_0}\,r\lambda^2
    +\frac{2a_+^2}{c_0^2}\,r\eta^2.
\label{eq:exact-low-rank-rate-bound}
\end{equation}

\end{corollary}

\begin{proof}
Applying Corollary~\ref{cor:exact-low-rank-score-tuning} with confidence level
$\delta/2$, $\rho=1/2$, and $\eta=\beta/T$ we obtain
\[
    \mathbb P_{A_0}\bigl(\mathcal S_T(\lambda,\eta)\bigr)
    \geq1-\frac{\delta}{2}.
\]
By Corollary~\ref{cor:exact-low-rank-curvature-simple} with confidence level
$\delta/2$,
\[
    \mathbb P_{A_0}\bigl(\mathcal R_T(c_0)\bigr)
    \geq1-\frac{\delta}{2}.
\]
Thus, the two events hold simultaneously with probability at least $1-\delta$.
On their intersection, Corollary~\ref{cor:exact-low-rank-oracle} gives
\[
    \big\|\widehat A_{\lambda,\eta}-A_0\big\|_F^2
    \leq
    \frac{72}{c_0}\,r\lambda^2
    +\frac{2}{c_0}\eta^2\big\|A_0B_{T,\eta}^{-1}\big\|_F^2.
\]
Moreover, on $\mathcal R_T(c_0)$,
$B_{T,\eta}^{-2}\preceq c_0^{-1}I_d$, and therefore
\[
    \big\|A_0B_{T,\eta}^{-1}\big\|_F^2
    =\operatorname{tr}\!\left(A_0^\top A_0B_{T,\eta}^{-2}\right)
    \leq\frac{\|A_0\|_F^2}{c_0}
    \leq\frac{a_+^2r}{c_0},
\]
which completes the proof.
\end{proof}

\noindent
Next, fix $\beta>0$ and $q>0$, and set
\[
\delta_T:=T^{-q},
\qquad
\eta_T:=\frac{\beta}{T}.
\]
Let $\lambda_T$ be defined by equality in \eqref{eq:exact-low-rank-rate-tuning}
with $\delta=\delta_T$. There exists a constant
\[
C=C(a_-,a_+,c_0,\beta,q)>0
\]
such that, for all sufficiently large $T$ satisfying
\[
T\geq K\bigl\{d_T+q\log T+\log 12\bigr\},
\]
we have
\[
\lambda_T^2\leq C\frac{d_T\log T}{T}.
\]
Consequently, after increasing $C$ if necessary,
\[
\mathbb P_{A_0}\!\left(
\bigl\|\widehat A_{\lambda_T,\eta_T}-A_0\bigr\|_F^2
\leq C\frac{r_Td_T\log T}{T}
\right)\geq 1-T^{-q}.
\]
In particular, along every joint asymptotic regime as $T\to\infty$ for which
\[
\frac{r_Td_T\log T}{T}\longrightarrow0,
\]
the estimator is Frobenius-norm consistent and
\[
\bigl\|\widehat A_{\lambda_T,\eta_T}-A_0\bigr\|_F^2
=
O_{\mathbb P}\!\left(\frac{r_Td_T\log T}{T}\right).
\]
Equivalently,
\[
\bigl\|\widehat A_{\lambda_T,\eta_T}-A_0\bigr\|_F
=
O_{\mathbb P}\!\left(
\sqrt{\frac{r_Td_T\log T}{T}}
\right).
\]
Thus, up to a factor $\sqrt{\log T}$ in Frobenius norm, the result has the
usual rank-$r_T$ matrix-estimation scaling
\[
\sqrt{\frac{r_Td_T}{T}}.
\]

\subsection{Block decomposition and auxiliary bounds}

The purpose of this subsection is to establish the blockwise estimates
needed for the proof of Theorem~\ref{thm:exact-low-rank-curvature}. Under the orthogonal decomposition
induced by the stable and null eigenspaces of \(A_0\), the empirical
covariance matrix \(C_T\) splits into the stable block \(S_T\), the
Brownian block \(Q_T\), and the cross block \(R_T\). We derive
concentration bounds for these three blocks and then
combine them, through a Schur-complement argument, to prove Theorem~\ref{thm:exact-low-rank-curvature}.

\label{subsec:exact-low-rank-blocks}

More precisely, under Assumption~\ref{ass:exact-low-rank}, let
\[
    Y_t:=P^\top X_t=
    \begin{pmatrix}U_t\\V_t\end{pmatrix},
\]
where $U_t\in\mathbb R^r$ and $V_t\in\mathbb R^m$. The orthogonal invariance
of Brownian motion gives independent coordinates
\begin{equation}
    dU_t^i=-a_iU_t^i\,dt+dB_t^i,
    \qquad U_0^i=0,
    \qquad i=1,\ldots,r,
\label{eq:stable-coordinate-model}
\end{equation}
whereas $V$ is an $m$-dimensional standard Brownian motion independent of $U$.
Consequently,
\begin{equation}
    P^\top C_TP
    =
    \begin{pmatrix}
        S_T & R_T\\
        R_T^\top & Q_T
    \end{pmatrix},
\label{eq:exact-low-rank-block-covariance}
\end{equation}
where
\begin{equation}
    S_T:=\frac1T\int_0^T U_tU_t^\top\,dt,
    \qquad
    R_T:=\frac1T\int_0^T U_tV_t^\top\,dt,
    \qquad
    Q_T:=\frac1T\int_0^T V_tV_t^\top\,dt.
\label{eq:exact-low-rank-block-definitions}
\end{equation}

\begin{lemma}
\label{lem:exact-low-rank-stable-block}
There exists a constant $C_S>0$, depending only on $a_-$, such that, for every
$x>0$,
\begin{equation}
    \mathbb P_{A_0}\!\left(
        \|S_T-\mathbb E_{A_0}S_T\|_{\mathrm{op}}
        >C_S\left\{
            \sqrt{\frac{r+x}{T}}+\frac{r+x}{T}
        \right\}
    \right)
    \leq2\exp(-x).
\label{eq:exact-low-rank-stable-concentration}
\end{equation}
Moreover,
\begin{equation}
    \lambda_{\min}\!\left(\mathbb E_{A_0}S_T\right)
    \geq
    \frac1{2a_+}-\frac1{4a_-^2T}.
\label{eq:exact-low-rank-stable-mean-lower}
\end{equation}
\end{lemma}

\begin{proof}
Let $\mathcal H_T:=L^2([0,T],dt/T)$ and for $u\in\mathbb S^{r-1}$, put
$G_u(t):=u^\top U_t$. Its covariance operator $\mathcal K_{u,T}$ on
$\mathcal H_T$ has kernel
\[
    \frac1T\operatorname{Cov}\bigl(G_u(s),G_u(t)\bigr).
\]
Since
\[
    \left|\operatorname{Cov}\bigl(G_u(s),G_u(t)\bigr)\right|
    \leq\frac{1}{2a_-}\exp(-a_-|t-s|),
\]
Schur's test gives
\[
    \|\mathcal K_{u,T}\|_{\mathrm{op}}\leq\frac1{a_-^2T},
    \qquad
    \operatorname{tr}(\mathcal K_{u,T})\leq\frac1{2a_-},
    \qquad
    \operatorname{tr}(\mathcal K_{u,T}^2)\leq\frac1{2a_-^3T}.
\]
The Karhunen-Lo\`eve expansion of $G_u$ in $\mathcal H_T$ and the Gaussian
quadratic-form inequality of \cite[Lemma~1]{LaurentMassart2000} therefore yield
\[
    \mathbb P_{A_0}\!\left(
        \left|u^\top(S_T-\mathbb E_{A_0}S_T)u\right|
        >C\left\{\sqrt{\frac{y}{T}}+\frac{y}{T}\right\}
    \right)
    \leq2\exp(-y)
\]
for a constant $C$ depending only on $a_-$. Applying this inequality on a
$1/4$-net of $\mathbb S^{r-1}$, whose cardinality is at most $9^r$, yields
\eqref{eq:exact-low-rank-stable-concentration} after enlarging the constant.

Finally, $\mathbb E_{A_0}S_T$ is diagonal and
\[
    \mathbb E_{A_0}(S_T)_{ii}
    =\frac1{2a_i}
     -\frac{1-\exp(-2a_iT)}{4a_i^2T}
    \geq\frac1{2a_+}-\frac1{4a_-^2T},
\]
which proves \eqref{eq:exact-low-rank-stable-mean-lower}.
\end{proof}

\begin{lemma}
\label{lem:exact-low-rank-brownian-block}
There is a universal constant $c_B>0$ such that, for every $x\geq1$,
\begin{equation}
    \mathbb P_{A_0}\!\left(
        \lambda_{\min}(Q_T)
        \geq c_B\frac{T}{m+x}
    \right)
    \geq1-\exp(-x).
\label{eq:exact-low-rank-brownian-lower}
\end{equation}
In particular, $Q_T$ is positive definite almost surely.
\end{lemma}

\begin{proof}
Brownian scaling gives
\[
    Q_T\stackrel{d}{=}T\mathsf G_m,
    \qquad
    \mathsf G_m:=\int_0^1B_uB_u^\top\,du,
\]
where $B$ is an $m$-dimensional standard Brownian motion. Its
Karhunen-Lo\`eve expansion yields
\[
    \mathsf G_m=\sum_{k\geq1}\mu_kz_kz_k^\top,
    \qquad
    \mu_k=\frac{1}{\pi^2(k-\tfrac12)^2},
\]
where $(z_k)_{k\geq1}$ are independent standard Gaussian vectors in
$\mathbb R^m$. Set $N:=\lceil16(m+x)\rceil$ and
$Z_N:=(z_1,\ldots,z_N)\in\mathbb R^{m\times N}$. Then
\[
    \mathsf G_m\succeq\mu_NZ_NZ_N^\top,
    \qquad
    \lambda_{\min}(\mathsf G_m)\geq\mu_Ns_{\min}(Z_N)^2.
\]
By the Gaussian smallest-singular-value inequality
\cite[Theorem~2.13]{DavidsonSzarek2001},
\[
    \mathbb P_{A_0}\!\left(
        s_{\min}(Z_N)<\sqrt N-\sqrt m-\sqrt{2x}
    \right)
    \leq \exp(-x).
\]
Since $N\geq16(m+x)$ and $m+x\geq1$,
\[
    \sqrt N-\sqrt m-\sqrt{2x}
    \geq(4-\sqrt3)\sqrt{m+x},
\]
whereas $N\leq17(m+x)$ and hence
\[
    \mu_N\geq\frac{1}{289\pi^2(m+x)^2}.
\]
Combining these estimates proves
\eqref{eq:exact-low-rank-brownian-lower} after decreasing the universal
constant. Positive definiteness follows by retaining the first $m$ terms in the
Karhunen-Lo\`eve expansion.
\end{proof}

\begin{lemma}
\label{lem:exact-low-rank-cross-block}
For every $x>0$,
\begin{equation}
    \mathbb P_{A_0}\!\left(
        \|R_TQ_T^{-1}R_T^\top\|_{\mathrm{op}}
        >C_R\frac{d+x}{T}
    \right)
    \leq2\exp(-x),
    \qquad
    C_R:=\frac{8\log9}{a_-^2}.
\label{eq:exact-low-rank-cross-bound}
\end{equation}
\end{lemma}

\begin{proof}
Again let $\mathcal H_T:=L^2([0,T],dt/T)$ and $\mathcal G_V:=\sigma\{V_t:0\leq t\leq T\}$. By
Lemma~\ref{lem:exact-low-rank-brownian-block}, $Q_T$ is invertible almost
surely. On this event, define
\[
    \phi_V(t):=Q_T^{-1/2}V_t\in\mathbb R^m,
    \qquad
    \Gamma_T:=R_TQ_T^{-1/2}\in\mathbb R^{r\times m}.
\]
The coordinates of $\phi_V$ are orthonormal in $\mathcal H_T$, and
\begin{equation}
    (\Gamma_T)_{ij}=\langle U^i,\phi_{V,j}\rangle_{\mathcal H_T},
    \qquad
    R_TQ_T^{-1}R_T^\top=\Gamma_T\Gamma_T^\top.
\label{eq:exact-low-rank-cross-representation}
\end{equation}
Conditional on $\mathcal G_V$, the functions
$\varphi_{V,1},\ldots,\varphi_{V,m}$ are deterministic. Since $U^i$ is
independent of $\mathcal G_V$, its conditional law given $\mathcal G_V$
coincides with its unconditional law. Hence, conditionally on
$\mathcal G_V$, the $i$th row of $\Gamma_T$ is a centered Gaussian vector.

For a deterministic $f\in\mathcal H_T$, Fubini theorem gives
\begin{equation}
    \langle U^i,f\rangle_{\mathcal H_T}
    =\int_0^T\left\{
        \frac1T\int_s^T \exp(-a_i(t-s))f(t)\,dt
    \right\}dB_s^i.
\label{eq:exact-low-rank-ito-functional}
\end{equation}
Let $\mathcal K_{i,T}$ denote the covariance operator of $U^i$ on
$\mathcal H_T$. The same Schur-test argument as in the proof of
Lemma~\ref{lem:exact-low-rank-stable-block} gives
\[
    \|\mathcal K_{i,T}\|_{\mathrm{op}}\leq\frac1{a_-^2T}.
\]
For $\theta\in\mathbb R^m$, set
\[
    f_{\theta,V}:=\sum_{j=1}^m\theta_j\phi_{V,j}.
\]
Since $\|f_{\theta,V}\|_{\mathcal H_T}=\|\theta\|_2$, the conditional
covariance matrix $\Sigma_{i,T}(V)$ of the $i$th row of $\Gamma_T$ satisfies
\[
    \theta^\top\Sigma_{i,T}(V)\theta
    =\langle\mathcal K_{i,T}f_{\theta,V},f_{\theta,V}\rangle_{\mathcal H_T}
    \leq\frac{\|\theta\|_2^2}{a_-^2T}.
\]
Thus,
\begin{equation}
    \Sigma_{i,T}(V)\preceq\frac1{a_-^2T}I_m
    \qquad\text{almost surely}.
\label{eq:exact-low-rank-row-covariance-bound}
\end{equation}
The rows of $\Gamma_T$ are conditionally independent because the coordinates
$U^1,\ldots,U^r$ are independent and independent of $V$.

For $u\in\mathbb S^{r-1}$ and $v\in\mathbb S^{m-1}$,
\eqref{eq:exact-low-rank-row-covariance-bound} gives, almost surely,
\[
    \mathbb P_{A_0}\!\left(
        |u^\top \Gamma_Tv|>t\mid\mathcal G_V
    \right)
    \leq2\exp\!\left(-\frac{a_-^2Tt^2}{2}\right).
\]
Let $\mathcal N_r$ and $\mathcal N_m$ be $1/4$-nets of
$\mathbb S^{r-1}$ and $\mathbb S^{m-1}$, respectively, with
\[
    |\mathcal N_r|\leq9^r,
    \qquad
    |\mathcal N_m|\leq9^m.
\]
The two-net inequality and a conditional union bound show, almost surely, that
\[
    \mathbb P_{A_0}\!\left(
        \|\Gamma_T\|_{\mathrm{op}}>2t\mid\mathcal G_V
    \right)
    \leq2\cdot9^d\exp\!\left(-\frac{a_-^2Tt^2}{2}\right).
\]
Take
\[
    t^2=\frac{2(d\log9+x)}{a_-^2T}.
\]
After taking expectations and using \eqref{eq:exact-low-rank-cross-representation},
this yields
\[
    \mathbb P_{A_0}\!\left(
        \|R_TQ_T^{-1}R_T^\top\|_{\mathrm{op}}
        >\frac{8(d\log9+x)}{a_-^2T}
    \right)
    \leq2\exp(-x).
\]
Since $d\log9+x\leq\log9(d+x)$, this proves
\eqref{eq:exact-low-rank-cross-bound}.
\end{proof}

\subsection{Proof of the curvature result}
\label{subsec:proof-exact-low-rank-curvature}

We now prove Theorem~\ref{thm:exact-low-rank-curvature} by combining the block estimates established in Lemmas~\ref{lem:exact-low-rank-stable-block}-\ref{lem:exact-low-rank-cross-block}.

\begin{proof}[Proof of Theorem~\ref{thm:exact-low-rank-curvature}]
Define the events
\[
    \mathcal E_S
    :=\left\{
        \|S_T-\mathbb E_{A_0}S_T\|_{\mathrm{op}}
        \leq\frac{\Delta}{6}
    \right\},
\]
\[
    \mathcal E_Q
    :=\left\{
        \lambda_{\min}(Q_T)
        \geq c_B\frac{T}{m+\ell_\delta}
    \right\},
\]
and
\[
    \mathcal E_R
    :=\left\{
        \|R_TQ_T^{-1}R_T^\top\|_{\mathrm{op}}
        \leq C_R\frac{d+\ell_\delta}{T}
    \right\}.
\]
The conditions \eqref{eq:curvature-stable-square-condition} and
\eqref{eq:curvature-stable-linear-condition} ensure that
\[
    C_S\left\{
        \sqrt{\frac{r+\ell_\delta}{T}}
        +\frac{r+\ell_\delta}{T}
    \right\}
    \leq\frac{\Delta}{6}.
\]
Therefore, Lemmas~\ref{lem:exact-low-rank-stable-block},
\ref{lem:exact-low-rank-brownian-block}, and
\ref{lem:exact-low-rank-cross-block} give
\[
    \mathbb P_{A_0}(\mathcal E_S^c)\leq2\exp(-\ell_\delta)=\frac{\delta}{3},
    \qquad
    \mathbb P_{A_0}(\mathcal E_Q^c)\leq \exp(-\ell_\delta)=\frac{\delta}{6},
    \qquad
    \mathbb P_{A_0}(\mathcal E_R^c)\leq2\exp(-\ell_\delta)=\frac{\delta}{3}.
\]
Consequently,
\begin{equation}
    \mathbb P_{A_0}(\mathcal E_S\cap\mathcal E_Q\cap\mathcal E_R)
    \geq1-\frac{5\delta}{6}
    \geq1-\delta.
\label{eq:exact-low-rank-good-event-probability}
\end{equation}
Work on the event
$\mathcal E:=\mathcal E_S\cap\mathcal E_Q\cap\mathcal E_R$. By
\eqref{eq:exact-low-rank-stable-mean-lower} and
\eqref{eq:curvature-init-condition},
\[
    \lambda_{\min}\!\left(\mathbb E_{A_0}S_T\right)
    \geq b-\frac{\Delta}{6}.
\]
Thus,
\[
    \lambda_{\min}(S_T)\geq b-\frac{\Delta}{3}.
\]
The cross-block condition \eqref{eq:curvature-cross-condition} gives
\[
    \|R_TQ_T^{-1}R_T^\top\|_{\mathrm{op}}
    \leq\frac{\Delta}{6},
\]
and hence
\begin{equation}
    \lambda_{\min}\!\left(
        S_T-R_TQ_T^{-1}R_T^\top
    \right)
    \geq b-\frac{\Delta}{2}=c_1.
\label{eq:exact-low-rank-schur-lower}
\end{equation}
Furthermore, \eqref{eq:curvature-brownian-condition} yields
\begin{equation}
    \lambda_{\min}(Q_T)\geq c_1.
\label{eq:exact-low-rank-Q-lower}
\end{equation}
On $\mathcal E$,
\[
    \|R_TQ_T^{-1}\|_{\mathrm{op}}^2
    \leq
    \|Q_T^{-1}\|_{\mathrm{op}}
    \|R_TQ_T^{-1}R_T^\top\|_{\mathrm{op}}.
\]
Using the defining bounds of $\mathcal E$ and
\eqref{eq:curvature-triangular-condition}, we obtain
\begin{equation}
    \|R_TQ_T^{-1}\|_{\mathrm{op}}\leq h.
\label{eq:exact-low-rank-triangular-bound}
\end{equation}
The Schur-complement factorization gives
\[
    P^\top C_TP
    =L_T
    \begin{pmatrix}
        S_T-R_TQ_T^{-1}R_T^\top & 0\\
        0 & Q_T
    \end{pmatrix}
    L_T^\top,
    \qquad
    L_T:=
    \begin{pmatrix}
        I_r & R_TQ_T^{-1}\\
        0 & I_m
    \end{pmatrix}.
\]
Since
\[
    \|L_T^{-1}\|_{\mathrm{op}}
    \leq1+\|R_TQ_T^{-1}\|_{\mathrm{op}},
\]
\eqref{eq:exact-low-rank-schur-lower},
\eqref{eq:exact-low-rank-Q-lower}, and
\eqref{eq:exact-low-rank-triangular-bound} imply
\[
    \lambda_{\min}(C_T)
    \geq\frac{c_1}{(1+h)^2}=c_0.
\]
Finally, $C_T+\eta I_d\succeq C_T$, which proves
\eqref{eq:exact-low-rank-curvature-probability} together with
\eqref{eq:exact-low-rank-good-event-probability}.
\end{proof}

\section{Numerical Study}
\label{sec:numerical-study}
In this section, we present two simulation experiments illustrating the empirical behaviour of the WNEE under different asymptotic regimes. The first experiment investigates the effect of increasing the dimension of the system, while the second focuses on the impact of increasing the observation horizon.
\subsection{Asymptotics in $d$}\label{subsec:Experiment_1}

The aim of this numerical experiment is to investigate the behaviour of the proposed estimator as the dimension of the system increases while the intrinsic rank of the interaction matrix remains fixed. We simulate trajectories from the Ornstein-Uhlenbeck process defined in \eqref{eq:intro-ou} with diffusion matrix $D=I_d$. The observation horizon is fixed at $T=20$ while  the true interaction matrix $A_0$ is generated as a rank-$5$ matrix satisfying Assumption~\ref{ass:spectral-structure}, with parameters $\overline{a}=5$ and $\overline{\kappa}=1$. The experiment is performed for $
d\in\{10,50,100,200,300,500\}$.
For each value of $d$, we generate $50$ independent trajectories and compute  the mean Frobenius error of the proposed estimators. In particular, we will compare the performance of four different estimators:
\begin{itemize}
    \item [1)] The classical maximum likelihood estimator (MLE), defined as the minimizer of \eqref{eq:contrast}.
    \item [2)]  The Ridge estimator, obtained by adding a Frobenius norm penalty to the minimization problem,
    \begin{equation*}
    \widehat{A}_{R}:= \underset{A\in\mathbb M_d}{\operatorname{argmin}} \left\{\mathcal L_T(A) + \eta_R \|A\|_F^2\right\}.
\end{equation*}
    \item[3)] The Nuclear estimator, based on Schatten-1 norm regularization,
    \begin{equation*}
    \widehat{A}_N := \underset{A\in\mathbb M_d}{\operatorname{argmin}} \left\{\mathcal L_T(A) + \lambda_N \|A\|_*\right\}.
\end{equation*}
    \item [4)] The proposed WNEE estimator defined in \eqref{eq:wnee}.
\end{itemize}
Since the theoretical framework assumes continuous-time observations, trajectories are generated numerically using the Euler-Maruyama scheme with fixed time step $\Delta_n=T/n = 10^{-3}$. Preliminary experiments showed that further reductions of the time step produced negligible changes in the estimation error while substantially increasing the computational cost. For each value of $d$,  and for each regularized estimator, the tuning parameters $\{ \eta_R, \lambda_N, (\eta,\lambda)\}$ are selected by cross-validation following a similar strategy to that proposed in \cite{CiolekMarushkevychPodolskij2025}.  Specifically, the observed trajectory is split into a training segment containing $70\%$ of the observations and a validation segment containing the remaining $30\%$. For each candidate value of the tuning parameters, the corresponding estimator is computed using the training segment. The optimal  parameters are then selected as those producing the estimator that minimizes  \eqref{eq:contrast} on the validation segment.

Figure~\ref{fig:Experiment1_Frobenius} reports the average Frobenius estimation error over the replications as a function of the dimension. The results reveal a clear separation between the classical MLE and the regularized estimators  as the dimension increases. For small and moderate values of $d$, the four estimators exhibit similar performance. However, as the dimension increases, the differences become substantial. The estimation error of the classical MLE grows rapidly with $d$, whereas all regularized estimators display a considerably slower increase. Among them, the proposed WNEE consistently achieves the smallest estimation error, and its advantage over both the Ridge and Nuclear estimators becomes increasingly pronounced for larger dimensions. These results suggest that combining weighted nuclear norm regularization with a Frobenius penalty provides a significant improvement in estimation accuracy in high-dimensional low-rank settings. To further illustrate the qualitative difference between the MLE and the WNEE, Figure~\ref{fig:Experiment1_Matrices} displays the true interaction matrix together with the corresponding estimates obtained by both methods for a representative realization with $d=300$. The corresponding estimation errors are displayed underneath each estimated matrix.

\begin{figure}[] 
    \centering
    \includegraphics[width=0.9\textwidth, height=9.4cm]{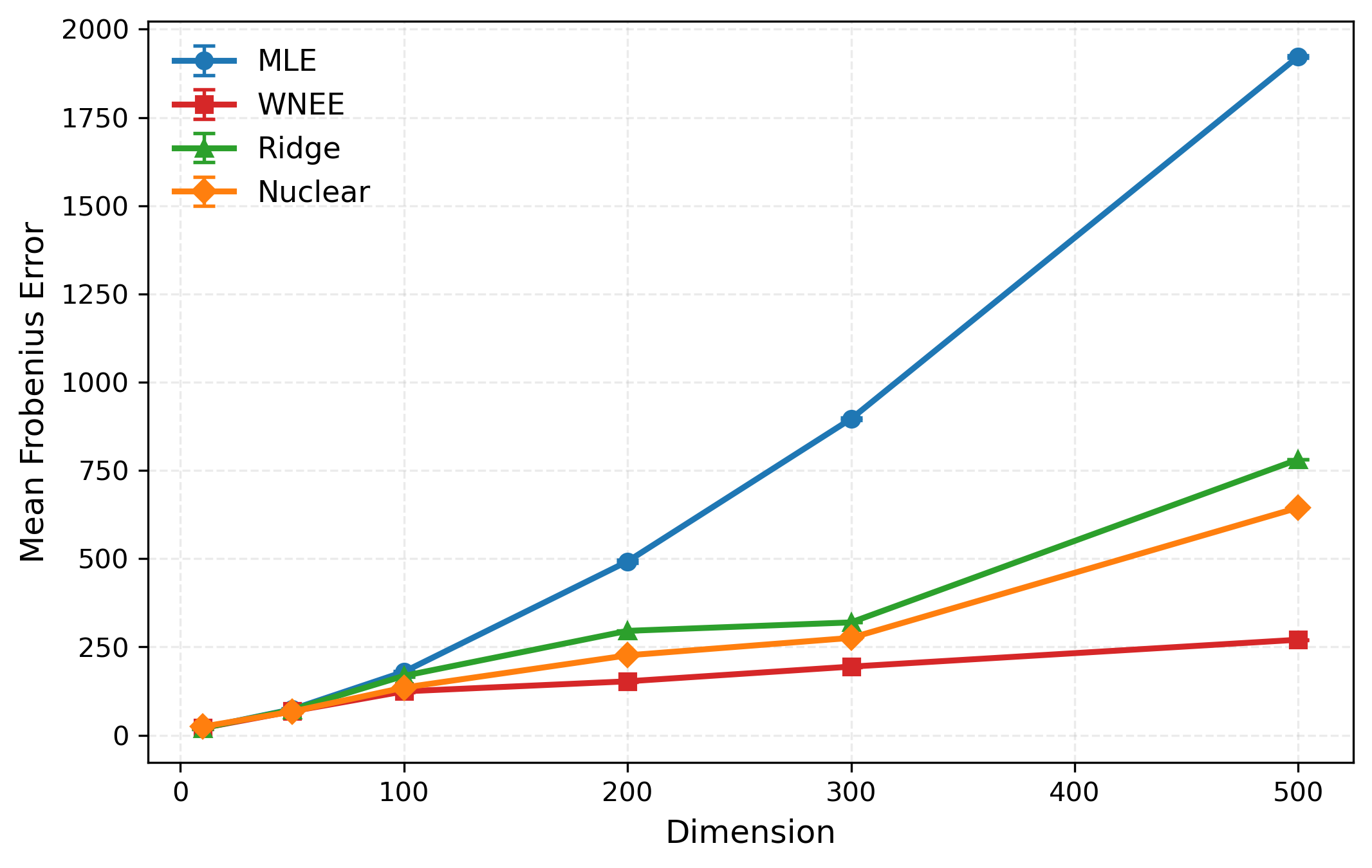} 
    \caption{Mean Frobenius estimation error (with standard deviation) over $50$ independent replications for the MLE, Ridge, Nuclear, and WNEE estimators as a function of the dimension $d$. The observation horizon is fixed at $T=20$ and the true interaction matrix has rank $r=5$.
    }
    \label{fig:Experiment1_Frobenius} 
\end{figure}

\begin{figure}[] 
    \centering
    \includegraphics[width=0.9\textwidth, height=9.4cm]{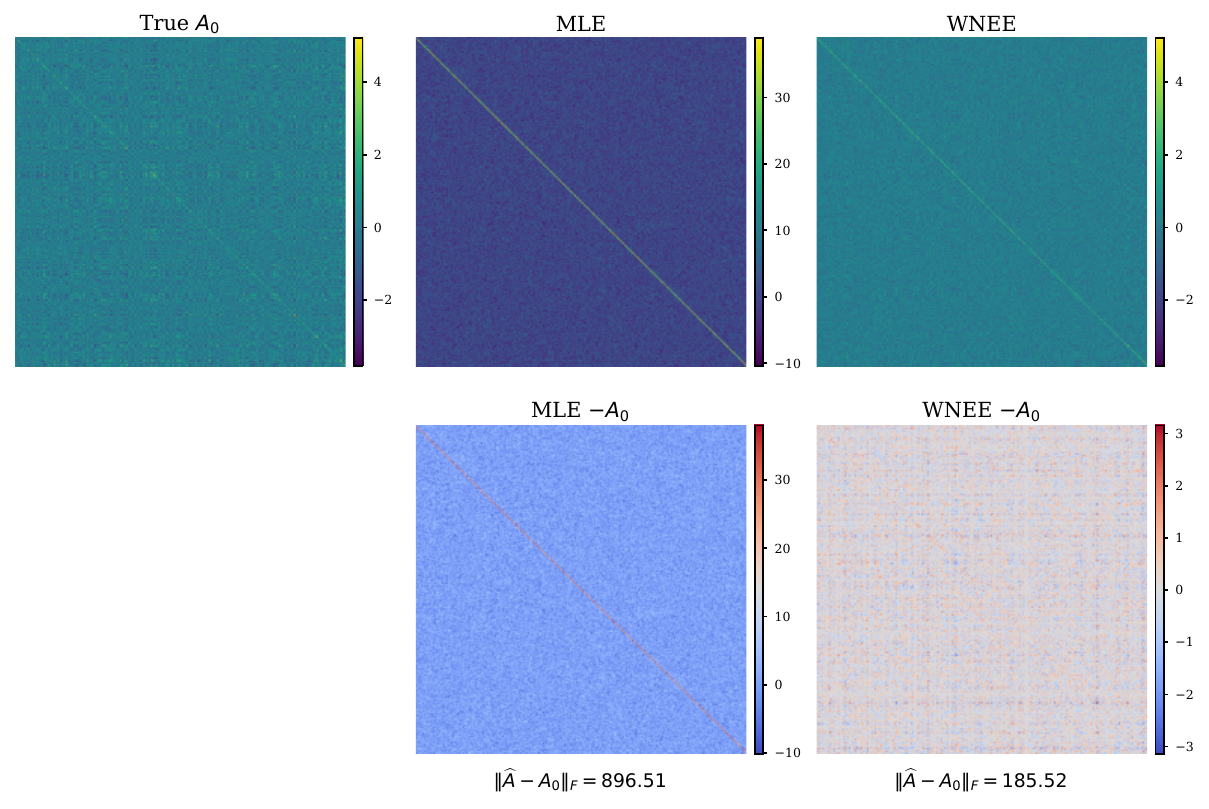} 
    \caption{Heatmaps of the true interaction matrix $A_0$ and the corresponding estimates obtained by the MLE and the proposed WNEE for a representative realization with $d=300$. The bottom row displays the estimation errors $\widehat{A}-A_0$, together with their Frobenius norms.
}
    \label{fig:Experiment1_Matrices} 
\end{figure}
\subsection{Asymptotics in $T$}

In the second numerical experiment, we investigate the behaviour of the proposed estimator as the observation horizon increases while the dimension of the system and the intrinsic rank of the interaction matrix remain fixed. Since the first experiment establishes the WNEE as the best-performing regularized estimator, we restrict the second experiment to the comparison between  the WNEE and the classical unregularized benchmark, the MLE. In this scenario, the trajectories are simulated from the OU process defined in \eqref{eq:intro-ou} under the same conditions in Experiment 1. In this case, the dimension is fixed at $d=300$, while the interaction matrix $A_0$ is generated as a rank-$5$ matrix satisfying Assumption~\ref{ass:spectral-structure} with parameters $\overline{a}=5$ and $\overline{\kappa}=1$. The observation horizon is varied over $T\in\{20,40,80,120\}.$
For each value of $T$, we generate $50$ independent trajectories and estimate $A_0$ using both the MLE and the WNEE. The remaining simulation settings, including the discretization scheme and the cross-validation procedure for selecting the regularization parameters, are identical to those of Experiment 1.

\begin{figure}[] 
    \centering
    \includegraphics[width=0.9\textwidth, height=7.2cm]{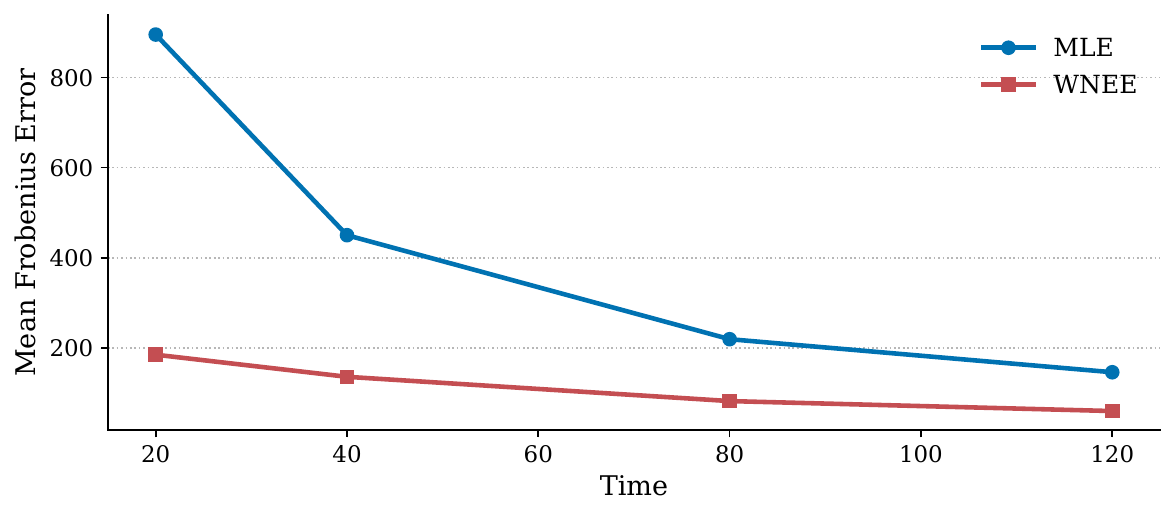} 
    \caption{Mean Frobenius estimation error (with standard deviation) over $50$ independent replications as a function of the time horizon  $T$. The dimension is fixed at $d=300$ while the true interaction matrix has rank $r=5$.
}
    \label{fig:Experiment2_Frobenius} 
\end{figure}

Figure~\ref{fig:Experiment2_Frobenius} reports the mean Frobenius  error over the  replications as a function of the observation horizon. As expected, the estimation error decreases for both estimators as more observations become available. Nevertheless, the proposed WNEE consistently achieves lower estimation errors than the classical MLE throughout the range of observation horizons considered.

\FloatBarrier
\bibliographystyle{plainnat} 
\bibliography{references}

\end{document}